\documentclass[11pt,twoside]{preprint}

\usepackage{hyperref}

\usepackage{times}
\usepackage{microtype}
\usepackage{amsmath}
\usepackage{amssymb}
\usepackage{tikz}

\usepackage{mathrsfs}
\usepackage{mhequ}
\usepackage{mhenvs}
\usepackage{mhsymb}
\usepackage[top=4cm, bottom=3.5cm, left=3.5cm, right=3.2cm]{geometry}

\usetikzlibrary{shapes.misc}
\usetikzlibrary{shapes.symbols}
\tikzset{
	dot/.style={circle,fill=black,draw=black,inner sep=0pt,minimum size=0.5mm},
	>=stealth,
	}
\makeatletter
\def\DeclareSymbol#1#2#3{\expandafter\gdef\csname MH@symb@#1\endcsname{\tikz[baseline=#2,scale=0.15]{#3}}}
\def\<#1>{\csname MH@symb@#1\endcsname}
\makeatother

\DeclareSymbol{X}{-2.4}{\node[dot] {};}
\DeclareSymbol{1}{0}{\draw[white] (-.4,0) -- (.4,0); \draw (0,0)  -- (0,1.2) node[dot] {};}
\DeclareSymbol{2}{0}{\draw (-0.5,1.2) node[dot] {} -- (0,0) -- (0.5,1.2) node[dot] {};}
\DeclareSymbol{3}{0}{\draw (0,0) -- (0,1.2) node[dot] {}; \draw (-.7,1) node[dot] {} -- (0,0) -- (.7,1) node[dot] {};}
\DeclareSymbol{31}{-3}{\draw (0,0) -- (0,-1) -- (1,0) node[dot] {}; \draw (0,0) -- (0,1.2) node[dot] {}; \draw (-.7,1) node[dot] {} -- (0,0) -- (.7,1) node[dot] {};}
\DeclareSymbol{30}{-3}{\draw (0,0) -- (0,-1); \draw (0,0) -- (0,1.2) node[dot] {}; \draw (-.7,1) node[dot] {} -- (0,0) -- (.7,1) node[dot] {};}
\DeclareSymbol{32}{-3}{\draw (0,0) -- (0,-1) -- (1,0) node[dot] {}; \draw (0,0) -- (0,-1) -- (-1,0) node[dot] {}; \draw (0,0) -- (0,1.2) node[dot] {}; \draw (-.7,1) node[dot] {} -- (0,0) -- (.7,1) node[dot] {};}
\DeclareSymbol{22}{-3}{\draw (0,0.3) -- (0,-1) -- (1,0) node[dot] {}; \draw (0,0.3) -- (0,-1) -- (-1,0) node[dot] {};\draw (-.7,1) node[dot] {} -- (0,0.3) -- (.7,1) node[dot] {};}
\DeclareSymbol{20}{-3}{\draw (0,0) -- (0,-1);\draw (-.7,1) node[dot] {} -- (0,0) -- (.7,1) node[dot] {};}
\DeclareSymbol{12}{-3}{\draw (0,0.3) -- (0,-1) -- (1,0) node[dot] {}; \draw (0,0.3) -- (0,-1) -- (-1,0) node[dot] {};\draw (-.7,1) node[dot] {} -- (0,0.3);}
\DeclareSymbol{10}{-3}{\draw (0,0.3) -- (0,-1);\draw (-.7,1) node[dot] {} -- (0,0.3);}
\DeclareSymbol{21}{-3}{\draw (0,0.3) -- (0,-1) -- (1,0) node[dot] {};\draw (-.7,1) node[dot] {} -- (0,0.3) -- (.7,1) node[dot] {};}

\newtheorem{assumption}[lemma]{Assumption}

\definecolor{darkred}{rgb}{0.9,0.1,0.1}

\def\comment#1{\ifthenelse{\isodd{\value{page}}}{\marginpar{\raggedright\scriptsize{\textcolor{darkred}{#1}}}}{\marginpar{\raggedleft\scriptsize{\textcolor{darkred}{#1}}}}}  

\newop{diag}

\let\D\CD

\def\TT{\mathscr{T}}

\def\bX{{\mathbf X}}

\def\XX{{\mathbb X}}

\def\E{{\mathbf E}}

\def\T{\mathbf{T}}
\def\/{|\!|\!|}

\def\${|\!|\!|}
\def\l|{\left|\!\left|\!\left|}
\def\r|{\right|\!\right|\!\right|}

\def\s{\mathfrak{s}}
\def\K{\mathfrak{K}}
\def\RR{\mathfrak{R}}

\def\sym{\mathrm{sym}}
\def\W{\mathbb{W}}
\def\PPi{\boldsymbol{\Pi}}
\def\Ren{\mathscr{R}}

\begin{document}

\title{Introduction to Regularity Structures}
\author{Martin Hairer}
\institute{The University of Warwick, \email{M.Hairer@Warwick.ac.uk}}

\maketitle

\begin{abstract}
These are short notes from a series of lectures given at the University of Rennes in June 2013, 
at the University of Bonn in July 2013,
at the XVII$^{\textit{th}}$ Brazilian School of Probability in Mambucaba in August 2013, and at ETH Zurich
in September 2013.
They give a concise overview of the
theory of regularity structures as exposed in the article \cite{Regular}.
In order to allow to focus on the conceptual aspects of the theory, many proofs are
omitted and statements are simplified. We focus on applying the theory to the problem
of giving a solution theory to the stochastic quantisation equations for the Euclidean
$\Phi^4_3$ quantum field theory.
\end{abstract}

\tableofcontents

\section{Introduction}

Very recently, a new theory of ``regularity structures'' was introduced \cite{Regular}, 
unifying various flavours of the theory of (controlled)
rough paths (including Gubinelli's theory of controlled rough paths \cite{Max},
as well as his branched rough paths \cite{Trees}), as well as the usual
Taylor expansions. While it has its roots in the theory of rough paths \cite{Lyons},
the main advantage of this new theory is that it is no longer tied to the
one-dimensionality of the time parameter, which makes it also suitable for the description of 
solutions to stochastic \textit{partial} differential equations, rather than just stochastic ordinary 
differential equations. The aim of this article is to give a concise
survey of the theory while focusing on the construction of the dynamical
$\Phi^4_3$ model. While the exposition aims to be reasonably self-contained (in particular
no prior knowledge of the theory of rough paths is assumed), most of the
proofs will only be sketched.

The main achievement of the theory of
regularity structures is that it allows to give a (pathwise!) meaning to ill-posed stochastic PDEs that arise naturally when trying to describe the macroscopic
behaviour of models from statistical mechanics near criticality. One
example of such an equation is the KPZ
equation arising as a natural model for one-dimensional interface motion \cite{KPZOrig,MR1462228,KPZ}:
\begin{equ}
\d_t h = \d_x^2 h + (\d_x h)^2 + \xi - C\;.
\end{equ}
Another example is the dynamical $\Phi^4_3$ model arising for example in the stochastic quantisation of Euclidean
quantum field theory \cite{ParisiWu,MR815192,AlbRock91,MR2016604,Regular}, as well as
a universal model for phase coexistence near the critical point \cite{MR1661764}:
\begin{equ}
\d_t \Phi = \Delta \Phi  + C \Phi - \Phi^3 + \xi\;.
\end{equ}
In both of these examples, $\xi$ formally denotes space-time white noise, $C$ is an arbitrary constant (which will actually turn out to be
infinite in some sense!), and we consider
a bounded square spatial domain with periodic boundary conditions. 
In the case of the dynamical $\Phi^4_3$ model, the spatial
variable has dimension $3$, while it has dimension $1$ in the case of the KPZ equation. While a full exposition
of the theory is well beyond the scope of this short introduction, we aim to give a concise overview to most of its
concepts. In most cases,
we will only state results in a rather informal way and give some ideas as to how the proofs work, focusing on
conceptual rather than technical issues. 
The only exception is the ``reconstruction theorem'', Theorem~\ref{theo:reconstruction} below,
which is the linchpin of the whole theory. Since its proof (or rather a slightly simplified version of it)
is relatively concise, we provide a fully self-contained version.
For precise statements and complete proofs of most of 
the results exposed here, we refer to the original article \cite{Regular}.

Loosely speaking, the type of well-posedness results that can be proven with the help of the 
theory of regularity structures can be formulated as follows.

\begin{theorem}
Let $\xi_\eps = \delta_\eps * \xi$ denote the regularisation of space-time 
white noise with a
compactly supported smooth mollifier $\delta_\eps$ that is scaled by $\eps$ in the spatial direction(s)
and by $\eps^2$ in the time direction. Denote by $h_\eps$ and $\Phi_\eps$ the solutions to
\begin{equs}
\d_t h_\eps &= \d_x^2 h_\eps + (\d_x h_\eps)^2 - C_\eps + \xi_\eps\;,\\
\d_t \Phi_\eps &= \Delta \Phi_\eps + \tilde C_\eps \Phi_\eps - \Phi_\eps^3 + \xi_\eps \;. 
\end{equs}
Then, there exist choices of constants $C_\eps$ and $\tilde C_\eps$ diverging as $\eps \to 0$, as
well as processes $h$ and $\Phi$ such that $h_\eps \to h$ and $\Phi_\eps \to \Phi$ in probability. 
Furthermore, while the constants $C_\eps$ and $\tilde C_\eps$ {\em do} depend crucially on the 
choice of mollifiers $\delta_\eps$, the limits $h$ and $\Phi$ do {\em not} depend on them.
\end{theorem}

\begin{remark}
We made a severe abuse of notation here since the space-time white noise
appearing in the equation for $h_\eps$ is on $\R \times \T^1$, while the
one appearing in the equation for $\Phi_\eps$ is on $\R \times \T^3$. (Here
we denote by $\T^n$ the $n$-dimensional torus.)
\end{remark}

\begin{remark}
We have not explicited the topology in which the convergence takes place in these
examples. In the case of the KPZ equation, one actually obtains convergence in
some space of space-time H\"older continuous functions. In the case of the 
dynamical $\Phi^4_3$ model, convergence takes place in some space of
space-time distributions. One caveat that also has to be dealt with in the
latter case is that the limiting process $\Phi$ may in principle 
explode in finite time for some instances of the driving noise.
\end{remark}

From a ``philosophical'' perspective, the theory of regularity structures is inspired by the 
theory of controlled rough paths \cite{Lyons,Max,LyonsStFlour}, so let us rapidly survey the main ideas of that theory.
The setting of the theory of controlled rough paths is the following. Let's say that we want to solve a
controlled differential equation of the type 
\begin{equ}[e:SDE]
dY = f(Y)\,dX(t)\;,
\end{equ}
 where $X \in \CC^\alpha$ is a rather rough function
(say a typical sample path for an $m$-dimensional Brownian motion). 
It is a classical result by Young \cite{Young} that the Riemann-Stieltjes
integral $(X,Y) \mapsto \int_0^\cdot Y\,dX$ makes sense as a continuous map from $\CC^\alpha \times \CC^\alpha$
into $\CC^\alpha$ if and only if $\alpha > {1\over 2}$. As a consequence, ``na\"\i ve'' approaches to 
a pathwise solution to \eref{e:SDE} are bound to fail if $X$ has the regularity of Brownian motion. 

The main idea is to exploit the a priori ``guess'' that solutions to \eref{e:SDE} should ``look like $X$ at small scales''.
More precisely, one would naturally expect the solution $Y$ to satisfy 
\begin{equ}[e:expY]
Y_t = Y_s + Y'_s X_{s,t} + \CO(|t-s|^{2\alpha})\;,
\end{equ}
where we wrote $X_{s,t}$ as a shorthand for the increment $X_t - X_s$. As a matter of fact, one would
expect to have such an expansion with $Y' = f(Y)$. Denote by $\CC^\alpha_X$ the space of pairs $(Y,Y')$
satisfying \eref{e:expY} for a given ``model path'' $X$.
It is then possible to simply ``postulate'' the values of the
integrals 
\begin{equ}[e:defXX]
\XX_{s,t} =: \int_s^t X_{s,r}\otimes dX_r\;,
\end{equ}
satisfying ``Chen's relations''
\begin{equ}[e:constr]
\XX_{s,t} - \XX_{s,u} -\XX_{u,t} =  X_{s,u}\otimes X_{u,t}\;,
\end{equ}
as well as the analytic bound $|\XX_{s,t}| \lesssim |t-s|^{2\alpha}$,
and to exploit this additional data to give a coherent definition of expressions of the type $\int Y\,dX$,
provided that the path $X$ is ``enhanced'' with its iterated integrals $\XX$ and $Y$ is a ``controlled path''
of the type \eref{e:expY}. See for example \cite{Max} for more information or \cite{RoughBurgers} for a concise exposition of
this theory.

Compare \eref{e:expY} to the fact that a function $f\colon \R \to \R$ 
is of class $\CC^\gamma$ with $\gamma \in (k,k+1)$ if for every $s \in \R$ there exist coefficients
$f_s^{(1)},\ldots,f_s^{(k)}$ such that 
\begin{equ}[e:expF]
f_t = f_s + \sum_{\ell=1}^k f_s^{(\ell)} (t-s)^\ell + \CO(|t-s|^\gamma)\;.
\end{equ}
Of course, $f_s^{(\ell)}$ is nothing but the $\ell$th derivative of $f$ at the point $s$, 
divided by $\ell!$. In this sense, one should really think of a controlled rough path $(Y,Y') \in \CC^\alpha_X$
as a $2\alpha$-H\"older continuous function, but with respect to a ``model'' determined
 by the function $X$, rather than by the 
usual Taylor polynomials.
This formal analogy between controlled rough paths and Taylor expansions suggests that it might be fruitful
to systematically investigate what are the ``right'' objects that could possibly take the place of Taylor
polynomials, while still retaining many of their nice properties.

\subsection*{Acknowledgements}

{\small
Financial support from the Leverhulme trust through a leadership award
is gratefully acknowledged.
}

\section{Definitions and the reconstruction operator}

The first step in such an endeavour is to set up an algebraic structure reflecting the properties of Taylor
expansions. First of all, such a structure should contain a vector space $T$ that will contain the
coefficients of our expansion. It is natural to assume that $T$ has a graded structure: $T = \bigoplus_{\alpha \in A} T_\alpha$,
for some set $A$ of possible ``homogeneities''. For example, in the case of the usual Taylor expansion \eref{e:expF},
it is natural to take for $A$ the set of natural numbers and to have $T_\ell$ contain the coefficients corresponding
to the derivatives of order $\ell$. In the case of controlled rough paths however, it is natural
to take $A = \{0,\alpha\}$, to have again $T_0$ contain the value of the function $Y$ at any time $s$,
and to have $T_\alpha$ contain the Gubinelli derivative $Y'_s$. This reflects the fact that the 
``monomial'' $t \mapsto X_{s,t}$ only vanishes at order $\alpha$ near $t = s$, while the usual monomials $t \mapsto (t-s)^\ell$
vanish at integer order $\ell$.

This however isn't the full algebraic structure describing Taylor-like expansions. Indeed, one of the characteristics of
Taylor expansions is that an expansion around some point $x_0$ can be re-expanded around any other point $x_1$ by writing
\begin{equ}[e:TaylorExp]
(x-x_0)^m = \sum_{k+\ell = m} {m!\over k! \ell!} (x_1 - x_0)^k\cdot (x-x_1)^\ell\;.
\end{equ}
(In the case when $x \in \R^d$, $k$, $\ell$ and $m$ denote multi-indices and $k! = k_1!\ldots k_d!$.) 
Somewhat similarly, in the case of controlled rough paths, we have the (rather trivial) identity
\begin{equ}[e:ControlExp]
X_{s_0,t} = X_{s_0, s_1}\cdot 1 + 1\cdot X_{s_1,t}\;.
\end{equ}
What is a natural abstraction of this fact? In terms of the coefficients of a ``Taylor expansion'', the operation of 
reexpanding around a different point is ultimately just a linear operation from $\Gamma \colon T \to T$, where the
precise value of the map $\Gamma$ depends on the starting point $x_0$, the endpoint $x_1$, and possibly also on the 
details of the particular ``model'' that we are considering. 
In view of the above examples, it is natural to impose furthermore that $\Gamma$ has the property that
if $\tau \in T_\alpha$, then $\Gamma \tau - \tau \in \bigoplus_{\beta < \alpha} T_\beta$. In other words, 
when reexpanding a homogeneous monomial around a different point, the leading order coefficient remains 
the same, but lower order monomials may appear. 

These heuristic considerations can be summarised in the following definition of an abstract object
we call a \textit{regularity structure}:

\begin{definition}
Let $A \subset \R$ be bounded from below and without accumulation point, and let 
$T = \bigoplus_{\alpha \in A} T_\alpha$ be a vector space graded by $A$ such that each $T_\alpha$ is a 
Banach space. Let furthermore $G$ be a group of continuous operators on $T$ such that, for every $\alpha \in A$,
every $\Gamma \in G$, and every $\tau \in T_\alpha$, one has $\Gamma \tau - \tau \in \bigoplus_{\beta < \alpha} T_\beta$.
The triple $\TT=(A,T,G)$ is called a \textit{regularity structure} with \textit{model space} $T$ and \textit{structure group} $G$.
\end{definition}

\begin{remark}
Given $\tau \in T$, we will write $\|\tau\|_\alpha$ for the norm of its component in $T_\alpha$.
\end{remark}

\begin{remark}
In \cite{Regular} it is furthermore assumed that $0 \in A$, $T_0 \approx \R$, and $T_0$ is invariant under $G$.
This is a very natural assumption which ensures that our regularity structure is at least sufficiently
rich to represent constant functions. 
\end{remark}

\begin{remark}
In principle, the set $A$ can be infinite. By analogy with the polynomials, it is then natural to 
consider $T$ as the set of all formal series of the form $\sum_{\alpha \in A}\tau_\alpha$, where only
finitely many of the $\tau_\alpha$'s are non-zero. This also dovetails nicely with the
particular form of elements in $G$. In practice however we will only ever work with finite subsets
of $A$ so that the precise topology on $T$ does not matter.
\end{remark}

At this stage, a regularity structure is a completely abstract object. It only becomes useful when endowed with
a \textit{model}, which is a concrete way of associating to any $\tau \in T$ and $x_0 \in \R^d$, the actual
``Taylor polynomial based at $x_0$'' represented by $\tau$. Furthermore, we want elements $\tau \in T_\alpha$ to
represent functions (or possibly distributions!) that ``vanish at order $\alpha$'' around the given point $x_0$.

Since we would like to allow $A$ to contain negative values and therefore allow elements in $T$ to represent actual
distributions, we need a suitable notion of ``vanishing at order $\alpha$''. We achieve this by considering the
size of our distributions, when tested against test functions that are localised around the given point $x_0$.
Given a test function $\phi$ on $\R^d$, we write $\phi_x^\lambda$ as a shorthand for
\begin{equ}
\phi_x^\lambda(y) = \lambda^{-d} \phi\bigl(\lambda^{-1}(y-x)\bigr)\;.
\end{equ}
Given $r >0$, we also denote by $\CB_r$ the set of all functions $\phi \colon \R^d \to \R$ such that 
$\phi \in \CC^r$ with $\|\phi\|_{\CC^r} \le 1$ that are furthermore supported in the unit ball around the origin.
With this notation, our definition of a model for a given regularity structure $\TT$ is as follows.

\begin{definition}\label{def:model}
Given a regularity structure $\TT$ and an integer $d \ge 1$, a \textit{model} for $\TT$ on $\R^d$ consists
of maps
\begin{equs}[2]
\Pi \colon \R^d &\to \CL\bigl(T, \CS'(\R^d)\bigr)&\qquad \Gamma\colon \R^d \times \R^d &\to G\\
x &\mapsto \Pi_x&\quad (x,y) &\mapsto \Gamma_{xy}
\end{equs}
such that $\Gamma_{xy}\Gamma_{yz} = \Gamma_{xz}$ and $\Pi_x \Gamma_{xy} = \Pi_y$. Furthermore, given $r > |\inf A|$, 
for any compact set $\K \subset \R^d$ and constant $\gamma > 0$, there exists a constant $C$ such that the bounds
\begin{equ}[e:bounds]
\bigl|\bigl(\Pi_x \tau\bigr)(\phi_x^\lambda)\bigr| \le C \lambda^{|\tau|} \|\tau\|_\alpha\;,\qquad
\|\Gamma_{xy}\tau\|_\beta \le C |x-y|^{\alpha-\beta} \|\tau\|_\alpha\;,
\end{equ}
hold uniformly over $\phi \in \CB_r$, $(x,y) \in \K$, $\lambda \in (0,1]$, 
$\tau \in T_\alpha$ with $\alpha \le \gamma$, and $\beta < \alpha$.
\end{definition}

\begin{remark}
In principle, test functions appearing in \eref{e:bounds} should be smooth. It turns out that if these bounds
hold for smooth elements of $\CB_r$, then $\Pi_x \tau$ can be extended canonically to allow any $\CC^r$
test function with compact support.
\end{remark}

\begin{remark}
The identity $\Pi_x \Gamma_{xy} = \Pi_y$ reflects the fact that $\Gamma_{xy}$ is the linear map that takes
an expansion around $y$ and turns it into an expansion around $x$. The first bound in \eref{e:bounds} states
what we mean precisely when we say that $\tau \in T_\alpha$ represents a term that vanishes at order $\alpha$.
(Note that $\alpha$ can be negative, so that this may actually not vanish at all!) 
The second bound in \eref{e:bounds} is very natural in view of both \eref{e:TaylorExp} and \eref{e:ControlExp}.
It states that when expanding a monomial of order $\alpha$ around a new point at distance $h$ from the old one,
the coefficient appearing in front of lower-order monomials of order $\beta$ is of order at most $h^{\alpha - \beta}$.
\end{remark}

\begin{remark}
In many cases of interest, it is natural to scale the different directions of $\R^d$ in a different way. 
This is the case for example when using the theory of regularity structures to build solution theories for
parabolic stochastic PDEs, in which case the time direction ``counts as'' two space directions.  
To deal with such  a situation, one can introduce a scaling $\s$ of $\R^d$, which is just a collection of
$d$ mutually prime strictly positive integers and to define $\phi_x^\lambda$ in such a  way that the $i$th
direction is scaled by $\lambda^{\s_i}$. In this case,  the Euclidean distance between two points should
be replaced everywhere by the corresponding scaled distance $|x|_\s = \sum_i |x_i|^{1/\s_i}$.
See also \cite{Regular} for more details.
\end{remark}

With these definitions at hand, it is then natural to define an equivalent in this context of the space
of $\gamma$-H\"older continuous functions in the following way.

\begin{definition}
Given a regularity structure $\TT$ equipped with a model $(\Pi,\Gamma)$ over $\R^d$, 
the space $\D^\gamma = \D^\gamma(\TT,\Gamma)$ is given by the set of functions $f\colon \R^d \to \bigoplus_{\alpha < \gamma} T_\alpha$
such that, for every compact set $\K$ and every $\alpha < \gamma$, the exists a constant $C$ with
\begin{equ}[e:defHolder]
\|f(x) - \Gamma_{xy} f(y)\|_{\alpha} \le C |x-y|^{\gamma-\alpha}
\end{equ}
uniformly over $x,y \in \K$.
\end{definition}

The most fundamental result in the theory of regularity structures then states that 
given $f \in \D^\gamma$ with $\gamma > 0$, there exists a \textit{unique} Schwartz distribution
$\CR f$ on $\R^d$ such that, for every $x \in \R^d$, $\CR f$ ``looks like $\Pi_x f(x)$ near $x$''.
More precisely, one has

\begin{theorem}\label{theo:reconstruction}
Let $\TT$ be a regularity structure as above and let $(\Pi,\Gamma)$ a model for $\TT$ on $\R^d$.
Then, there exists a unique linear map $\CR\colon \D^\gamma \to \CS'(\R^d)$ such that
\begin{equ}[e:boundRf]
\bigl|\bigl(\CR f - \Pi_x f(x)\bigr)(\phi_x^\lambda)\bigr| \lesssim \lambda^\gamma\;,
\end{equ}
uniformly over $\phi \in \CB_r$ and $\lambda$ as before, and locally uniformly in $x$.
\end{theorem}

\begin{proof}
The proof of the theorem relies on the following fact. Given any $r > 0$ (but finite!), there 
exists a function $\phi \colon \R^d \to \R$ with the following properties:
\begin{claim}
\item[(1)] The function $\phi$ is of class $\CC^r$ and has compact support.
\item[(2)] For every polynomial $P$ of degree $r$, there exists a polynomial $\hat P$ of degree $r$ 
such that, for every $x \in \R^d$, one has $\sum_{y \in \Z^d} \hat P(y) \phi(x-y) = P(x)$.
\item[(3)] One has $\int \phi(x)\phi(x-y)\,dx = \delta_{y,0}$ for every $y \in \Z^d$.
\item[(4)] There exist coefficients $\{a_k\}_{k \in \Z^d}$
such that $2^{-d/2}\phi(x/2) = \sum_{k \in \Z^d} a_k \phi(x-k)$.
\end{claim}
The existence of such a function $\phi$ is highly non-trivial. This is actually equivalent to the existence
of a  wavelet basis consisting of $\CC^r$ functions with compact support, a proof of which
was first obtained by Daubechies in her seminal article \cite{MR951745}. From now on, we take the existence of
such a function $\phi$ as a given for some $r > |\inf A|$. 
We also set $\Lambda^n = 2^{-n} \Z^d$ and, for $y \in \Lambda^n$, we set
$\phi^n_y(x) = 2^{nd/2}\phi(2^n(x-y))$. Here, the normalisation is chosen in such a way that the set $\{\phi^n_y\}_{y\in \Lambda^n}$ is again orthonormal in $L^2$. We then denote by $V_n \subset \CC^r$ the linear span of $\{\phi^n_y\}_{y\in \Lambda^n}$,
so that, by the property (4) above, one has $V_0 \subset V_1 \subset V_2 \subset \ldots$. We furthermore denote by
$\hat V_n$ the $L^2$-orthogonal complement of $V_{n-1}$ in $V_{n}$, so that
$V_n = V_0 \oplus \hat V_1 \oplus\ldots\oplus \hat V_n$. In order to keep notations compact, it will also be convenient
to define the coefficients $a^n_k$ with $k \in \Lambda^n$ by $a^n_k = a_{2^n k}$.

With these notations at hand, we then define a sequence of linear operators $\CR^n \colon \D^\gamma \to \CC^r$ by
\begin{equ}
\bigl(\CR^n f\bigr)(y) = \sum_{x \in \Lambda^n} \bigl(\Pi_x f(x)\bigr)(\phi_x^n)\,\phi_x^n(y)\;.
\end{equ}
We claim that there then exists a Schwartz distribution $\CR f$ such that, 
for every compactly supported test function $\psi$ of class $\CC^r$, one has $\scal{\CR^n f, \psi} \to \bigl(\CR f\bigr)(\psi)$,
and that $\CR f$ furthermore satisfies the properties stated in the theorem.

Let us first consider the size of the components of $\CR^{n+1} f - \CR^n f$ in $V_n$. Given $x \in \Lambda^n$, we make
use of properties (3-4), so that
\begin{equs}
\scal{\CR^{n+1} f - \CR^n f, \phi_x^n} &= \sum_{k \in \Lambda^{n+1}} a^n_k \scal{\CR^{n+1} f, \phi_{x+k}^{n+1}} - \bigl(\Pi_x f(x)\bigr)(\phi_x^n) \\
&= \sum_{k \in \Lambda^{n+1}} a^n_k \bigl(\Pi_{x+k} f(x+k)\bigr)(\phi_{x+k}^{n+1})  - \bigl(\Pi_x f(x)\bigr)(\phi_x^n) \\
&= \sum_{k \in \Lambda^{n+1}} a^n_k \bigl(\bigl(\Pi_{x+k} f(x+k)\bigr)(\phi_{x+k}^{n+1})  - \bigl(\Pi_x f(x)\bigr)(\phi_{x+k}^{n+1})\bigr)\\
&= \sum_{k \in \Lambda^{n+1}} a^n_k \bigl(\Pi_{x+k}\bigl(f(x+k) - \Gamma_{x+k,x} f(x)\bigr)\bigr)(\phi_{x+k}^{n+1})\;,
\end{equs}
where we used the algebraic relations between $\Pi_x$ and $\Gamma_{xy}$ to obtain the last identity.
Since only finitely many of the coefficients $a_k$ are non-zero, it follows from the definition of $\D^\gamma$ that
for the non-vanishing terms in this sum we have the bound
\begin{equ}
\|f(x+k) - \Gamma_{x+k,x} f(x)\|_\alpha \lesssim 2^{-n(\gamma-\alpha)}\;,
\end{equ}
uniformly over $n \ge 0$ and $x$ in any compact set. Furthermore, for any $\tau \in T_\alpha$, it follows from
the definition of a model that one has the bound
\begin{equ}
\bigl|\bigl(\Pi_x \tau\bigr)(\phi_x^n)\bigr| \lesssim 2^{-\alpha n - {nd \over 2}}\;,
\end{equ}
again uniformly over $n \ge 0$ and $x$ in any compact set. Here, the additional factor $2^{- {nd \over 2}}$ comes
from the fact that the functions $\phi_x^n$ are normalised in $L^2$ rather than $L^1$. Combining these two bounds, we
immediately obtain that
\begin{equ}[e:bounddiff]
\bigl|\scal{\CR^{n+1} f - \CR^n f, \phi_x^n}\bigr| \lesssim 2^{-\gamma n - {nd \over 2}}\;,
\end{equ}
uniformly over $n \ge 0$ and $x$ in compact sets. Take now a test function $\psi \in \CC^r$ with compact support
and let us try to estimate $\scal{\CR^{n+1} f - \CR^n f, \psi}$. Since $\CR^{n+1} f - \CR^n f \in V_{n+1}$,
we can decompose it into a part $\delta \CR^n f \in V_n$ and a part $\hat \delta \CR^n f\in \hat V_{n+1}$ and estimate both parts separately. Regarding the part in $V_n$, we have
\begin{equ}[e:bounddeltaR]
\bigl|\scal{\delta \CR^n f , \psi}\bigr| = \Bigl|\sum_{x \in \Lambda^{n+1}} \scal{\delta \CR^n f , \phi_x^n}\scal{\phi_x^n,\psi}\Bigr| \lesssim 2^{-\gamma n - {nd \over 2}} \sum_{x \in \Lambda^{n+1}} \bigl|\scal{\phi_x^n,\psi}\bigr|\;,
\end{equ}
where we made use of the bound \eref{e:bounddiff}. At this stage we use the fact that, due to the boundedness of $\psi$,
we have $\bigl|\scal{\phi_x^n,\psi}\bigr| \lesssim 2^{-nd/2}$. Furthermore, thanks to the boundedness of the support of $\psi$,
the number of non-vanishing terms appearing in this sum is bounded by $2^{nd}$, so that we eventually obtain the bound
\begin{equ}[e:firstBound]
\bigl|\scal{\delta \CR^n f , \psi}\bigr| \lesssim 2^{-\gamma n}\;.
\end{equ}
Regarding the second term, we use the standard fact coming from wavelet analysis \cite{MR1228209} that a basis of 
$\hat V_{n+1}$ can be obtained in the same way as the basis of $V_n$,
but replacing the function $\phi$ by functions $\hat \phi$ from some finite set $\Phi$. In other words, 
 $\hat V_{n+1}$ is the linear span of $\{\hat \phi_x^n\}_{x\in \Lambda^n; \hat \phi\in \Phi}$. Furthermore, 
 as a consequence of property (2), the functions
 $\hat \phi \in \Phi$ all have the property that 
 \begin{equ}[e:killPoly]
 \int \hat \phi(x)\,P(x)\,dx = 0\;, 
 \end{equ}
for any polynomial $P$ of degree
 less or equal to $r$. In particular, this shows that one has the bound
\begin{equ}
|\scal{\hat \phi_x^n, \psi}| \lesssim 2^{- {nd \over 2} - nr}\;.
\end{equ}
As a consequence,  one has
\begin{equ}
\bigl|\scal{\hat \delta \CR^n f , \psi}\bigr|  = \Bigl|\sum_{x\in \Lambda^n \atop \hat \phi \in \Phi}\scal{\CR^{n+1} f , \hat\phi_x^n}\scal{\hat \phi_x^n, \psi}\Bigr| \lesssim 2^{- {nd \over 2} - nr} \Bigl|\sum_{x\in \Lambda^n \atop \hat \phi \in \Phi}\scal{\CR^{n+1} f , \hat\phi_x^n}\Bigr|\;.
\end{equ}
At this stage, we note that, thanks to the definition of $\CR^{n+1}$ and the bounds on the model $(\Pi,\Gamma)$,
we have $|\scal{\CR^{n+1} f , \hat\phi_x^n}| \lesssim 2^{-{nd\over 2} - \alpha_0 n}$, where $\alpha_0 = \inf A$,
so that $\bigl|\scal{\hat\delta \CR^n f , \psi}\bigr| \lesssim 2^{-nr - \alpha_0 n}$. Combining this with
\eref{e:firstBound}, we see that one has indeed $\CR^n f \to \CR f$ for some Schwartz distribution $\CR f$.

It remains to show that the bound \eref{e:boundRf} holds. For this, given a distribution $\eta \in \CC^{\alpha}$ for 
some $\alpha > -r$, we first introduce the notation
\begin{equ}
\CP_n \eta = \sum_{x \in \Lambda^n} \eta(\phi_x^n)\,\phi_x^n\;,\qquad 
\hat \CP_n \eta = \sum_{\hat \phi \in \Phi} \sum_{x \in \Lambda^n} \eta(\hat \phi_x^n)\,\hat \phi_x^n\;.
\end{equ}
We also choose an integer value $n \ge 0$ such that $2^{-n} \sim \lambda$
and we write 
\begin{equs}
\CR f - \Pi_x f(x) &= \CR^n f - \CP_n \Pi_x f(x) + \sum_{m \ge n} \bigl(\CR^{m+1}f - \CR^m f - \hat \CP_m \Pi_x f(x)\bigr)\\
&= \CR^n f - \CP_n \Pi_x f(x) + \sum_{m \ge n} \bigl(\hat \delta\CR^{m}f - \hat \CP_m \Pi_x f(x)\bigr)+ \sum_{m \ge n} \delta\CR^{m}f\;.\qquad\quad\label{e:decomposition}
\end{equs}
We then test these terms against $\psi_x^\lambda$ and we estimate the resulting terms separately.
For the first term, we have the identity
\begin{equ}[e:firstTerm]
\bigl(\CR^n f - \CP_n \Pi_x f(x)\bigr)(\psi_x^\lambda) = \sum_{y \in \Lambda^n} \bigl(\Pi_y f(y) - \Pi_x f(x)\bigr)(\phi_y^n)\,\scal{\phi_y^n, \psi_x^\lambda}\;.
\end{equ}
We have the bound
$|\scal{\phi_y^n, \psi_x^\lambda}|\lesssim \lambda^{-d} 2^{-dn/2} \sim 2^{dn/2}$. Since one furthermore has $|y-x| \lesssim \lambda$ for
all non-vanishing terms in the sum, one also has similarly to before 
\begin{equ}[e:diffPixy]
|\bigl(\Pi_y f(y) - \Pi_x f(x)\bigr)(\phi_y^n)| 
\lesssim \sum_{\alpha < \gamma} \lambda^{\gamma - \alpha} 2^{-{dn \over 2}- \alpha n} \sim 2^{-{dn \over 2}- \gamma n}\;.
\end{equ}
Since only finitely many (independently of $n$) terms contribute to the sum in \eref{e:firstTerm}, it is indeed bounded
by a constant proportional to $2^{-\gamma n} \sim \lambda^\gamma$ as required. 

We now turn to the second term in \eref{e:decomposition}, where we consider some fixed value $m \ge n$.
We rewrite this term very similarly to before as
\begin{equ}
\bigl(\hat \delta\CR^m f - \hat \CP_m \Pi_x f(x)\bigr)(\psi_x^\lambda) = \sum_{\hat \phi \in \Phi}\sum_{y,z} \bigl(\Pi_y f(y) - \Pi_x f(x)\bigr)(\phi_y^{m+1})\,\scal{\phi_y^{m+1}, \hat \phi_z^m}\,\scal{\hat \phi_z^m, \psi_x^\lambda}\;,
\end{equ} 
where the sum runs over $y \in \Lambda^{m+1}$ and $z \in \Lambda^m$.
This time, we use the fact that
by the property \eqref{e:killPoly} of the wavelets $\hat \phi$, one has the bound
\begin{equ}[e:boundWavelet]
|\scal{\hat \phi_z^m, \psi_x^\lambda}| \lesssim \lambda^{-d-r} 2^{-rm- {md \over 2}}\;,
\end{equ}
and the $L^2$-scaling implies that $|\scal{\phi_y^{m+1}, \hat \phi_z^m}| \lesssim 1$.
Furthermore, for each $z\in \Lambda^m$, only finitely many elements $y \in \Lambda^{m+1}$ contribute
to the sum, and these elements all satisfy $|y-z| \lesssim 2^{-m}$.
Bounding the first factor as in \eref{e:diffPixy} and using the fact that there are of the order of
$\lambda^d 2^{md}$ terms contributing for every fixed $m$, we thus see that the contribution of the second term
in \eqref{e:decomposition} is bounded by
\begin{equ}
\sum_{m \ge n} \lambda^d 2^{md} \sum_{\alpha < \gamma} \lambda^{\gamma - \alpha - d-r} 2^{-dm- \alpha m - rm}
\sim \sum_{\alpha < \gamma} \lambda^{\gamma-\alpha-r} \sum_{m \ge n} 2^{-\alpha m -rm}
\sim \lambda^{\gamma}\;.
\end{equ}

For the last term in \eref{e:decomposition}, we combine \eref{e:bounddeltaR} 
with the bound $|\scal{\phi_y^m, \psi_x^\lambda}|\lesssim \lambda^{-d} 2^{-dm/2}$
and the fact that there are of the order of $\lambda^d 2^{-md}$ terms appearing
in the sum \eref{e:bounddeltaR} to conclude that the $m$th summand is bounded by 
a constant proportional to $2^{-\gamma m}$. Summing over $m$ yields again the desired bound
and concludes the proof.
\end{proof}

\begin{remark}
Note that the space $\D^\gamma$ depends crucially on the choice of model $(\Pi,\Gamma)$.
As a consequence, the reconstruction operator $\CR$ itself also depends on that choice.
However, the map $(\Pi, \Gamma, f) \mapsto \CR f$ turns out to be locally Lipschitz 
continuous provided that the distance between $(\Pi, \Gamma, f)$ and $(\bar \Pi, \bar \Gamma, \bar f)$ is given by
the smallest constant $\rho$ such that
\begin{equs}
\|f(x) - \bar f(x) - \Gamma_{xy} f(y) + \bar \Gamma_{xy} \bar f(y)\|_{\alpha} &\le \rho |x-y|^{\gamma-\alpha}\;,\\
\bigl|\bigl(\Pi_x \tau - \bar \Pi_x \tau\bigr)(\phi_x^\lambda)\bigr| &\le \rho \lambda^{\alpha} \|\tau\|\;,\\
\|\Gamma_{xy}\tau - \bar \Gamma_{xy}\tau\|_\beta &\le \rho |x-y|^{\alpha-\beta} \|\tau\|\;.
\end{equs}
Here, in order to obtain bounds on $\bigl(\CR f - \bar \CR \bar f\bigr)(\psi)$ for some smooth compactly supported 
test function $\psi$, the above bounds should hold uniformly for $x$ and $y$ in a neighbourhood of the support of $\psi$.
The proof that this stronger continuity property also holds is actually crucial when showing that sequences of
solutions to mollified equations all converge to the same limiting object. However, its proof is somewhat more
involved which is why we chose not to give it here.
\end{remark}

\begin{remark}\label{rem:continuousModel}
In the particular case where $\Pi_x \tau$ happens to be a continuous function for every $\tau \in T$
(and every $x \in \R^d$), $\CR f$ is also a continuous function and one has the identity
\begin{equ}[e:formulaRf]
\bigl(\CR f\bigr)(x) = \bigl(\Pi_x f(x)\bigr)(x)\;.
\end{equ}
This can be seen from the fact that
\begin{equ}
\bigl(\CR f\bigr)(y) = \lim_{n \to \infty} \bigl(\CR^n f\bigr)(y) = \lim_{n \to \infty} \sum_{x \in \Lambda^n} \bigl(\Pi_x f(x)\bigr)(\phi_x^n)\,\phi_x^n(y)\;.
\end{equ}
Indeed, our assumptions imply that the function $(x,z) \mapsto \bigl(\Pi_x f(x)\bigr)(z)$ is 
jointly continuous and since the non-vanishing terms in the above sum satisfy $|x-y| \lesssim 2^{-n}$,
one has $2^{dn/2}\bigl(\Pi_x f(x)\bigr)(\phi_x^n) \approx \bigl(\Pi_y f(y)\bigr)(y)$ for large $n$.
Since furthermore $\sum_{x \in \Lambda^n} \phi_x^n(y) = 2^{dn/2}$, the claim follows.
\end{remark}

\section{Examples of regularity structures}

\subsection{The polynomial structure}
\label{sec:polynomial}

It should by now be clear how the structure given by the usual Taylor polynomials fits into this framework.
A natural way of setting it up is to take for $T$ the space of all abstract polynomials in $d$ commuting variables,
denoted by $X_1,\ldots, X_d$, and to postulate that $T_k$ consists of the linear span of monomials of degree $k$.
As an abstract group, the structure group $G$ is then given by $\R^d$ endowed with addition as its group operation, 
which acts onto $T$ via $\Gamma_h X^k = (X-h)^k$, where $h \in \R^d$ and we use the notation $X^k$ as a shorthand
for $X_1^{k_1}\cdots X_d^{k_d}$ for any multiindex $k$.

The canonical polynomial model is then given by
\begin{equ}
\bigl(\Pi_x X^k\bigr)(y) = (y-x)^k\;,\qquad \Gamma_{xy} = \Gamma_{y-x}\;.
\end{equ}
We leave it as an exercise to the reader to verify that this does indeed satisfy the bounds and relations
of Definition~\ref{def:model}.

In the particular case of the canonical polynomial model and for $\gamma \not \in \N$, the spaces $\D^\gamma$ then 
coincide precisely with the
usual H\"older spaces $\CC^\gamma$. In the case of integer values, this should be interpreted as bounded
functions for $\gamma = 0$, Lipschitz
continuous functions for $\gamma = 1$, etc.

\subsection{Controlled rough paths}

Let us see now how the theory of controlled rough paths can be reinterpreted in the light of this theory.
For given $\alpha \in ({1\over 3},{1\over 2})$ and $n \ge 1$, we can define a regularity structure $\TT$ by setting 
$A = \{\alpha-1, 2\alpha-1, 0, \alpha\}$. We furthermore take for $T_0$ a copy of $\R$ with unit vector $\one$,
for $T_\alpha$ and $T_{\alpha-1}$ a copy of $\R^n$ with respective unit vectors $\W_j$ and $\Xi_j$,
and for $T_{2\alpha-1}$ a copy of $\R^{n \times n}$ with unit vectors $\W_j\Xi_i$.
The structure group $G$ is taken to be isomorphic to $\R^n$ and, for $x \in \R^n$, it acts on $T$ via
\begin{equ}
\Gamma_x \one = \one\;,\qquad \Gamma_x \Xi_i = \Xi_i\;,\qquad \Gamma_x \W_i = \W_i - x^i \one\;,\qquad \Gamma_x (\W_j\Xi_i) = \W_j\Xi_i - x^j \Xi_i\;.
\end{equ}
Let now $\bX = (X,\XX)$ be an $\alpha$-H\"older continuous rough path with values in $\R^n$.
In other words, the functions $X$ and $\XX$ are as in the introduction, satisfying the relation
\eref{e:constr} and the analytic bounds $|X_t - X_s| \lesssim |t-s|^\alpha$, $|\XX_{s,t}| \lesssim |t-s|^{2\alpha}$.  
It turns out that this defines a model for $\TT$ in the following way (recall that 
$X_{s,t}$ is a shorthand for $X_t - X_s$):

\begin{lemma}
Given an $\alpha$-H\"older continuous rough path $\bX$, one can define a model for $\TT$ on $\R$
by setting $\Gamma_{su} = \Gamma_{X_{s,u}}$ and
\begin{equs}[2]
\bigl(\Pi_s \one\bigr)(t) &= 1\;,&\qquad \bigl(\Pi_s \W_j\bigr)(t) &= X^j_{s,t} \\
\bigl(\Pi_s \Xi_j\bigr)(\psi) &= \int \psi(t)\,dX^j_t\;,&\qquad 
\bigl(\Pi_s \W_j\Xi_i\bigr)(\psi) &= \int \psi(t)\,d\XX^{i,j}_{s,t}\;.
\end{equs}
Here, both integrals are perfectly well-defined Riemann integrals, with the differential in the second case 
taken with respect to the variable $t$. Given a controlled rough path $(Y,Y') \in \CC^\alpha_X$ as in \eref{e:expY}, this then defines an
element $\hat Y \in \D^{2\alpha}$ by setting
\begin{equ}
\hat Y(s) = Y(s)\,\one + Y'_i(s)\,\W_i\;,
\end{equ}
with summation over $i$ implied.
\end{lemma}

\begin{proof}
We first check that the algebraic properties of Definition~\ref{def:model} are satisfied. It is clear that $\Gamma_{su}\Gamma_{ut} = \Gamma_{st}$ and that $\Pi_s\Gamma_{su} \tau = \Pi_u\tau$ for $\tau \in \{\one, \W_j, \Xi_j\}$.
Regarding $\W_j\Xi_i$, we differentiate Chen's relations \eref{e:constr} which yields the identity
\begin{equ}
d\XX^{i,j}_{s,t} = d\XX^{i,j}_{u,t} + X^i_{s,u} \, dX^j_{t}\;.
\end{equ}
The last missing algebraic relation then follows at once. The required analytic bounds follow immediately from the definition
of the rough path space $\CD^\alpha$.

Regarding the function $\hat Y$ defined in the statement, we have
\begin{equs}
\|\hat Y(s) - \Gamma_{su} \hat Y(u)\|_0 &= |Y(s) - Y(u) + Y'_i(u) X^i_{s,u}|\;,\\
\|\hat Y(s) - \Gamma_{su} \hat Y(u)\|_\alpha &= |Y'(s) - Y'(u)|\;,
\end{equs}
so that the condition \eref{e:defHolder} with $\gamma = 2\alpha$ does indeed 
coincide with the definition of a controlled rough path given in the introduction.
\end{proof}

In this context, the reconstruction theorem allows us to define an integration operator
with respect to $W$. We can formulate this as follows where one should really think
of $Z$ as providing a consistent definition of what one means by $\int Y\,dX^j$.

\begin{lemma}\label{lem:RP}
In the same context as above, let
$\alpha \in ({1\over 3}, {1\over 2})$, and consider $\hat Y \in \D^{2\alpha}$ built as above
from a controlled rough path. Then, the map $\hat Y \Xi_i$ given by
\begin{equ}
\bigl(\hat Y \Xi_j\bigr)(s) =  Y(s)\,\Xi_j + Y'_i(s)\,\W_i\Xi_j
\end{equ}
belongs to $\D^{3\alpha-1}$. Furthermore, there exists a function $Z$ such that, for every
smooth test function $\psi$, one has
\begin{equ}
\bigl(\CR \hat Y \Xi_j\bigr)(\psi) = \int \psi(t)\,dZ(t)\;,
\end{equ}
and such that $Z_{s,t} = Y(s)\,X^j_{s,t} + Y'_i(s)\, \XX^{i,j}_{s,t} + \CO(|t-s|^{3\alpha})$.
\end{lemma}

\begin{proof}
The fact that $\hat Y \Xi_i \in \D^{3\alpha-1}$ is an immediate consequence of the definitions. 
Since $\alpha > {1\over 3}$ by assumption, we can apply the reconstruction theorem to it, from
which it follows that there exists a unique distribution $\eta$ such that, if $\psi$ is a 
smooth compactly supported test function, one has
\begin{equ}
\eta(\psi_s^\lambda) = \int \psi_s^\lambda(t) Y(s)\,dX^j_t + \int \psi_s^\lambda(t) Y'_i(s)\,d\XX^{i,j}_{s,t} + \CO(\lambda^{3\alpha-1})\;.
\end{equ}
By a simple approximation argument, it turns out that one can take for $\psi$ the indicator function
of the interval $[0,1]$, so that
\begin{equ}
\eta(\one_{[s,t]}) = Y(s)\,X^j_{s,t} + Y'_i(s)\,\XX^{i,j}_{s,t} + \CO(|t-s|^{3\alpha})\;.
\end{equ}
Here, the reason why one obtains an exponent $3\alpha$ rather than $3\alpha -1$ is that 
it is really $|t-s|^{-1}\one_{[s,t]}$ that scales like an approximate $\delta$-distribution as $t \to s$.
\end{proof}

\begin{remark}
Using the formula \eref{e:formulaRf}, it is straightforward to verify that if
$X$ happens to be a smooth function and $\XX$ is defined from $X$ via \eref{e:defXX},
but this time viewing it as a definition for the left hand side, with the right hand side given
by a usual Riemann integral, then the function $Z$ constructed in Lemma~\ref{lem:RP}
coincides with the usual Riemann integral of $Y$ against $X^j$.
\end{remark}

\subsection{A classical result from harmonic analysis}
\label{sec:classical}

The considerations above suggest that a very natural space of distributions is obtained in the following way.
For some $\alpha > 0$, we denote by $\CC^{-\alpha}$ the space of all Schwartz distributions $\eta$ such that
$\eta$ belongs to the dual of $\CC^r$ with $r > \alpha$ some integer and such that
\begin{equ}
\bigl|\eta(\phi_x^\lambda)\bigr| \lesssim \lambda^{-\alpha}\;, 
\end{equ}  
uniformly over all $\phi \in \CB_r$ and $\lambda \in (0,1]$, and locally uniformly in $x$.
Given any compact set $\K$, the best possible constant such that the above bound holds uniformly 
over $x \in \K$ yields a seminorm. The collection of these seminorms endows $\CC^{-\alpha}$ with
a Fr\'echet space structure.

\begin{remark}
It turns out that the space $\CC^{-\alpha}$ is independent of the choice of $r$ in the definition given above,
which justifies the notation. Different values of $r$ give raise to equivalent seminorms.
\end{remark}

\begin{remark}
In terms of the scale of classical Besov spaces, the space $\CC^{-\alpha}$ is a local 
version of $\CB^{-\alpha}_{\infty,\infty}$. It is in some sense the largest space of distributions that is 
invariant under the scaling $\phi(\cdot) \mapsto \lambda^{-\alpha} \phi(\lambda^{-1} \cdot)$,
see for example \cite{Bourgain}.
\end{remark}

It is then a classical result in the ``folklore'' of harmonic analysis that the product
extends naturally to $\CC^{-\alpha} \times \CC^\beta$ into $\CS'(\R^d)$ if and only if 
$\beta > \alpha$. The reconstruction theorem yields a straightforward proof of the ``if''
part of this result:

\begin{theorem}\label{thm:barrier}
There is a continuous bilinear map $B \colon \CC^{-\alpha} \times \CC^\beta \to \CS'(\R^d)$
such that $B(f,g) = fg$ for any two continuous functions $f$ and $g$.
\end{theorem}

\begin{proof}
Assume from now on that $\xi \in \CC^{-\alpha}$ for some $\alpha > 0$ and that $f \in \CC^\beta$
for some $\beta > \alpha$. 
We then build a regularity structure $\TT$ in the following way. For the set $A$,
we take $A = \N \cup (\N-\alpha)$ and
for $T$, we set $T = V \oplus W$,
where each one of the spaces $V$ and $W$ is a copy of the polynomial model in $d$ commuting variables
constructed in Section~\ref{sec:polynomial}.
We also choose $\Gamma$ as in the canonical
model, acting simultaneously on each of the two instances.

As before, we denote by $X^k$ the canonical basis vectors in $V$. We also use the suggestive notation
``$\Xi X^k$'' for the corresponding basis vector in $W$, but we postulate that $\Xi X^k \in T_{\alpha + |k|}$ 
rather than $\Xi X^k \in T_{|k|}$.
Given any distribution $\xi \in \CC^{-\alpha}$, we then define a model $(\Pi^\xi,\Gamma)$,
where $\Gamma$ is as in the canonical model, while $\Pi^\xi$ acts as
\begin{equ}
\bigl(\Pi^\xi_x  X^k\bigr)(y) = (y-x)^k \;,\qquad
\bigl(\Pi^\xi_x  \Xi X^k\bigr)(y) = (y-x)^k \xi(y)\;,
\end{equ}
with the obvious abuse of notation in the second expression. 
It is then straightforward to verify that 
$\Pi_y = \Pi_x\circ \Gamma_{xy}$ and that the relevant analytical bounds are satisfied, so that this is indeed a model.

Denote now by $\CR^\xi$ the reconstruction map associated to the model $(\Pi^\xi,\Gamma)$ and, for $f \in \CC^\beta$,
denote by $F$ the element in $\D^\beta$ given by the local Taylor expansion of $f$ of order $\beta$ at each point.
Note that even though the space $\D^\beta$ does in principle depend on the choice of model, in our situation
$F \in \D^\beta$ for any choice of $\xi$. 
It follows immediately from the definitions that the map $x \mapsto \Xi F(x)$ belongs to $\D^{\beta - \alpha}$
so that, provided that $\beta > \alpha$, one can apply the reconstruction operator to it. This suggests that the
multiplication operator we are looking for can be defined as
\begin{equ}
B(f,\xi) = \CR^\xi \bigl(\Xi F\bigr)\;.
\end{equ}
By Theorem~\ref{theo:reconstruction}, 
this is a jointly continuous map from $\CC^\beta \times \CC^{-\alpha}$ into $\CS'(\R^d)$, provided that 
$\beta > \alpha$.
If $\xi$ happens to be a smooth function, then it follows immediately from Remark~\ref{rem:continuousModel} that 
$B(f,\xi) = f(x)\xi(x)$, so that $B$ is indeed the requested continuous extension of the usual product.
\end{proof}

\begin{remark}
As a consequence of \eref{e:boundRf}, it is actually easy to show that $B \colon \CC^{-\alpha} \times \CC^\beta \to \CC^{-\alpha}$.
\end{remark}

\section{Products and composition by smooth functions}

One of the main purposes of the theory presented here is to give a robust way to multiply distributions
(or functions with distributions) that goes beyond the barrier illustrated by Theorem~\ref{thm:barrier}.
Provided that our functions / distributions are represented as elements in $\D^\gamma$ for some
model and regularity structure, we can multiply their ``Taylor expansions'' pointwise, provided that
we give ourselves a table of multiplication on $T$.

It is natural to consider products with the following properties. Here, given a regularity structure,
we say that a subspace
$V \subset T$ is a \textit{sector} if it is invariant under the action of the structure group $G$
and if it can furthermore be written as $V = \bigoplus_{\alpha \in A} V_\alpha$
with $V_\alpha \subset T_\alpha$.

\begin{definition}
Given a regularity structure $(T,A,G)$ and two sectors $V, \bar V \subset T$, a \textit{product}
on $(V,\bar V)$ is a bilinear map $\star \colon V \times \bar V \to T$ such that, for any $\tau \in V_\alpha$
and $\bar \tau \in \bar V_\beta$, one has $\tau \star \bar \tau \in T_{\alpha + \beta}$ and such that,
for any element $\Gamma \in G$, one has $\Gamma(\tau \star \bar \tau) = \Gamma \tau \star \Gamma \bar \tau$. 
\end{definition}

\begin{remark}
The condition that homogeneities add up under multiplication is very natural bearing in mind the case
of the polynomial regularity structure. The second condition is also very natural since it merely states that
if one reexpands the product of two ``polynomials'' around a different point, one should obtain the same result
as if one reexpands each factor first and then multiplies them together.
\end{remark}

Given such a product, we can ask ourselves when the pointwise product of an element $\D^{\gamma_1}$
with an element in $\D^{\gamma_2}$ again belongs to some $\D^\gamma$. In order to answer this question, 
we introduce the notation $\D_\alpha^\gamma$ to denote those elements $f \in \D^\gamma$ such that 
furthermore 
\begin{equ}
f(x) \in T_\alpha^+ \equiv \bigoplus_{\beta \ge \alpha} T_\beta\;,
\end{equ}
for every $x$. With this notation at hand, 
it is not too difficult to verify that one has the following result:

\begin{theorem}\label{theo:mult}
Let $f_1 \in \D^{\gamma_1}_{\alpha_1}(V)$, $f_2 \in \D^{\gamma_2}_{\alpha_2}(\bar V)$, and let $\star$
be a product on $(V,\bar V)$. Then, the function $f$ given by $f(x) = f_1(x) \star f_2(x)$ belongs to $\D_{\alpha}^\gamma$
with
\begin{equ}[e:formulaGamma]
\alpha = \alpha_1 +\alpha_2\;,\qquad \gamma = (\gamma_1 + \alpha_2)\wedge (\gamma_2 + \alpha_1)\;.
\end{equ} 
\end{theorem}

\begin{proof}
It is clear that $f(x) \in \bigoplus_{\beta > \alpha} T_\beta$, so it remains to show that 
it belongs to $\D^\gamma$. 
Furthermore, since we are only interested in showing that $f_1\star f_2 \in \D^\gamma$, we
 discard all of the components in $T_\beta$ for $\beta \ge \gamma$. 

By the properties of the product $\star$, it remains to obtain a bound of the type 
\begin{equ}
\|\Gamma_{xy} f_1(y) \star \Gamma_{xy} f_2(y) - f_1(x) \star f_2(x)\|_\beta \lesssim |x-y|^{\gamma-\beta} \;.
\end{equ}
By adding and subtracting suitable terms, we obtain
\begin{equs}
\|\Gamma_{xy} f(y) - f(x)\|_\beta
&\le  \|\bigl(\Gamma_{xy} f_1(y) - f_1(x)\bigr)\star \bigl(\Gamma_{xy} f_2(y) - f_2(x)\bigr)\|_\beta \label{e:splitProd}\\
&\quad + \|\bigl(\Gamma_{xy} f_1(y) - f_1(x)\bigr)\star f_2(x)\|_\beta + \|f_1(x)\star \bigl(\Gamma_{xy} f_2(y) - f_2(x)\bigr)\|_\beta\;. 
\end{equs}
It follows from the properties of the product $\star$ that the first term in \eref{e:splitProd} 
is bounded by a constant times
\begin{equs}
\sum_{\beta_1 + \beta_2 =\beta} &\|\Gamma_{xy} f_1(y) - f_1(x)\|_{\beta_1} \|\Gamma_{xy} f_2(y) - f_2(x)\|_{\beta_2} \\
&\lesssim \sum_{\beta_1 + \beta_2 =\beta} \|x-y\|^{\gamma_1-\beta_1}\|x-y\|^{\gamma_2-\beta_2} 
\lesssim \|x-y\|^{\gamma_1+\gamma_2-\beta}\;.
\end{equs}
Since $\gamma_1 + \gamma_2 \ge \gamma$, this bound is as required.
The second term is bounded by a constant times
\begin{equ}
\sum_{\beta_1 + \beta_2 = \beta} \|\Gamma_{xy} f_1(y) - f_1(x)\|_{\beta_1} \|f_2(x)\|_{\beta_2} \lesssim \|x-y\|^{\gamma_1-\beta_1} \,\one_{\beta_2 \ge \alpha_2} \lesssim
\|x-y\|^{\gamma_1+\alpha_2-\beta}\;,
\end{equ}
where the second inequality uses the identity $\beta_1 + \beta_2 = \beta$. 
Since $\gamma_1 + \alpha_2 \ge \gamma$, this bound is again of the required type. 
The last term is bounded similarly by reversing 
the roles played by $f_1$ and $f_2$.
\end{proof}

\begin{remark}
It is clear that the formula \eref{e:formulaGamma} for $\gamma$ is optimal in general as can be seen from
the following two ``reality checks''. First, consider the case of the polynomial model
and take $f_i \in \CC^{\gamma_i}$. In this case, the truncated Taylor series $F_i$ for $f_i$ belong to $\D_0^{\gamma_i}$. 
It is clear that in this case, the product cannot be expected to have better regularity than $\gamma_1 \wedge \gamma_2$ in
general, which is indeed what \eref{e:formulaGamma} states.
The second reality check comes from the example of Section~\ref{sec:classical}.
In this case, one has $F \in \D_0^\beta$, while the constant function $x \mapsto \Xi$ belongs
to $\D_{-\alpha}^\infty$ so that, according to \eref{e:formulaGamma}, one expects their product
to belong to $\D_{-\alpha}^{\beta-\alpha}$, which is indeed the case.
\end{remark}

It turns out that if we have a product on a regularity structure, then in many cases this also naturally
yields a notion of composition with smooth functions. Of course, one could in general not expect to be able to
compose a smooth function with a distribution of negative order. As a matter of fact, we will only define
the composition of smooth functions with elements in some $\D^\gamma$ for which it is guaranteed that the
reconstruction operator yields a continuous function. One might think at this case that this would yield
a triviality, since we know of course how to compose arbitrary continuous function. The subtlety is that we
would like to design our composition operator in such a way that the result is again an element of $\D^\gamma$.

For this purpose, we say that a given sector $V \subset T$ is \textit{function-like} if 
$\alpha < 0 \Rightarrow V_\alpha = 0$ and if $V_0$ is one-dimensional. (Denote the unit vector of $V_0$ by $\one$.)
We will furthermore always assume that our models are \textit{normal} in the sense that $\bigl(\Pi_x \one\bigr)(y) = 1$.
I this case, it turns out that if $f \in \D^\gamma(V)$, then $\CR f$ is a continuous function and one has the identity
$\bigl(\CR f\bigr)(x) = \scal{\one,f(x)}$, where we denote by $\scal{\one,\cdot}$ the element in the dual of $V$ which
picks out the prefactor of $\one$.

Assume now that we are given a regularity structure with a function-like sector $V$ and a product
$\star \colon V \times V \to V$. For any smooth function $G \colon \R \to \R$ and any $f \in \D^\gamma(V)$
with $\gamma > 0$, we can then \textit{define} $G(f)$ to be the $V$-valued function given by
\begin{equ}
\bigl(G\circ f\bigr)(x) = \sum_{k \ge 0} {G^{(k)}(\bar f(x)) \over k!} \tilde f(x)^{\star k}\;,
\end{equ}
where we have set
\begin{equ}
\bar f(x) = \scal{\one, f(x)}\;,\qquad \tilde f(x) = f(x) - \bar f(x)\one\;.
\end{equ}
Here, $G^{(k)}$ denotes the $k$th derivative of $G$ and $\tau^{\star k}$ denotes the $k$-fold product
$\tau\star\cdots \star \tau$. We also used the usual conventions $G^{(0)} = G$ and $\tau^{\star 0} = \one$.

Note that as long as $G$ is $\CC^\infty$, this expression is well-defined. Indeed, by assumption, there exists
some $\alpha_0 > 0$ such that $\tilde f(x) \in T_{\alpha_0}^+$. By the properties of the product, this
implies that one has $\tilde f(x)^{\star k} \in T_{k \alpha_0}^+$. As a consequence, when considering the
component of $G \circ f$ in $T_\beta$ for $\beta < \gamma$,  the only terms that give a contribution
are those with $k < \gamma / \alpha_0$. Since we cannot possibly hope in general that $G \circ f \in \D^{\gamma'}$
for some $\gamma' > \gamma$, this is all we really need.

It turns out that if $G$ is sufficiently regular, then the map $f \mapsto G\circ f$ enjoys similarly 
nice continuity properties to what we are used to from classical H\"older spaces. The following result
is the analogue in this context to the well-known fact that the composition of a $\CC^\gamma$ function with a sufficiently smooth function $G$ is again of class $\CC^\gamma$.

\begin{proposition}
In the same setting as above, provided that $G$ is of class $\CC^k$ with $k > \gamma / \alpha_0$,
the map $f \mapsto G \circ f$ is continuous from $\D^{\gamma}(V)$ into itself. If $k > \gamma / \alpha_0 + 1$,
then it is locally Lipschitz continuous.
\end{proposition}

The proof of this result can be found in \cite{Regular}. It is somewhat lengthy, but ultimately 
rather straightforward.

\subsection{A simple example}

A very important remark is that even 
if both $\CR f_1$ and $\CR f_2$ happens to be continuous functions, this does
\textit{not} in general imply that $\CR(f_1 \star f_2)(x) = (\CR f_1)(x) \, (\CR f_2)(x)$!
For example, fix $\kappa < 0$ and consider the regularity structure given by $A = (-2\kappa,-\kappa, 0)$,
with each $T_\alpha$ being a copy of $\R$ given by $T_{-n\kappa} = \scal{\Xi^n}$. 
We furthermore take for $G$ the trivial group. This regularity structure comes with an obvious product
by setting $\Xi^m \star \Xi^n = \Xi^{m+n}$ provided that $m+n \le 2$.

Then, we could for example take as a model for $\TT = (T,A,G)$:
\begin{equ}
\bigl(\Pi_x \Xi^0\bigr)(y) = 1\;,\quad 
\bigl(\Pi_x \Xi\bigr)(y) = 0\;,\quad \label{e:modelNonStandard}
\bigl(\Pi_x \Xi^2\bigr)(y) = c\;,
\end{equ}
where $c$ is an arbitrary constant. Let furthermore
\begin{equ}
F_1(x) = f_1(x) \Xi^0 + f_1'(x) \Xi\;,\qquad
F_2(x) = f_2(x) \Xi^0 + f_2'(x) \Xi\;.
\end{equ}
Since our group $G$ is trivial, one has
$F_i \in \D^\gamma$ provided that each of the $f_i$ belongs to $\D^\gamma$
and each of the $f_i'$  belongs to $\D^{\gamma + \kappa}$. (And one has $\gamma + \kappa < 1$.)
One furthermore has the identity $\bigl(\CR F_i\bigr)(x) = f_i(x)$. 

However, the pointwise product is given by
\begin{equ}
\bigl(F_1\star F_2\bigr)(x) = f_1(x)f_2(x) \Xi^0 + \bigl(f_1'(x) f_2(x) + f_2'(x) f_1(x)\bigr) \Xi + f_1'(x) f_2'(x) \Xi^2\;,
\end{equ}
which by Theorem~\ref{theo:mult} belongs to $\D^{\gamma - \kappa}$. Provided that $\gamma > \kappa$,
one can then apply the reconstruction operator to this product and  one obtains
\begin{equ}
\CR \bigl(F_1\star F_2\bigr)(x) = f_1(x)f_2(x) + c f_1'(x) f_2'(x)\;,
\end{equ}
which is obviously different from the pointwise product $\CR F_1 \cdot \CR F_2$.

How should this be interpreted? For $n > 0$, we could have defined a model $\Pi^{(n)}$ by
\begin{equ}
\bigl(\Pi_x \Xi^0\bigr)(y) = 1\;,\quad 
\bigl(\Pi_x \Xi\bigr)(y) = \sqrt{2c} \sin(nx)\;,\quad 
\bigl(\Pi_x \Xi^2\bigr)(y) = 2c \sin^2(nx)\;.
\end{equ}
Denoting by $\CR^{(n)}$ the corresponding reconstruction operator, we have the identity
\begin{equ}
\bigl(\CR^{(n)}F_i\bigr)(x) = f_i(x) + \sqrt{2c} f_i'(x)  \sin(nx)\;,
\end{equ}
as well as $\CR^{(n)}(F_1 \star F_2) = \CR^{(n)}F_1 \cdot \CR^{(n)} F_2$. 
As a model, the model $\Pi^{(n)}$ actually converges to the limiting model $\Pi$ defined in \eref{e:modelNonStandard}.
As a consequence of the continuity of the reconstruction operator, this implies that
\begin{equ}
\CR^{(n)}F_1 \cdot \CR^{(n)} F_2 = \CR^{(n)}(F_1 \star F_2) \to \CR (F_1 \star F_2) \neq \CR F_1 \cdot \CR F_2\;,
\end{equ}
which is of course also easy to see ``by hand''. This shows that in some cases, the ``non-standard'' models
as in \eref{e:modelNonStandard} can be interpreted as limits of ``standard'' models for which the usual
rules of calculus hold. Even this is however not always the case. 

\section{Schauder estimates and admissible models}
\label{sec:Schauder}

One of the reasons why the theory of regularity structures is very successful at providing detailed
descriptions of the small-scale features of solutions to semilinear (S)PDEs is that it comes with
very sharp Schauder estimates. Recall that the classical Schauder estimates state that if $K\colon \R^d \to \R$
is a kernel that is smooth everywhere, except for a singularity at the origin that is (approximately) homogeneous
of degree $\beta - d$ for some $\beta > 0$, then the operator $f \mapsto K * f$ maps $\CC^\alpha$ into $\CC^{\alpha +\beta}$
for every $\alpha \in \R$, except for those values for which $\alpha + \beta \in \N$. (See for example \cite{MR1459795}.)

It turns out that similar Schauder estimates hold in the context of general regularity structures in the sense
that it is in general possible to build an operator $\CK \colon \D^\gamma \to \D^{\gamma+\beta}$ with the
property that $\CR \CK f = K * \CR f$. Of course, such a statement can only be true if our regularity structure
contains not only the objects necessary to describe $\CR f$ up to order $\gamma$, but also those required
to describe $K * \CR f$ up to order $\gamma + \beta$. What are these objects? At this stage, it might be useful
to reflect on the effect of the convolution of a singular function (or distribution) with $K$. 

Let us assume for a moment that $f$ is also smooth everywhere, except at some point $x_0$. It is then 
straightforward to convince ourselves that $K * f$ is also smooth everywhere, except at $x_0$. Indeed,
for any $\delta > 0$, we can write $K = K_\delta + K_\delta^c$, where $K_\delta$ is supported in a ball of radius
$\delta$ around $0$ and $K_\delta^c$ is a smooth function. Similarly, we can decompose $f$ as $f = f_\delta + f_\delta^c$,
where $f_\delta$ is supported in  a $\delta$-ball around $x_0$ and $f_\delta^c$ is smooth. Since the convolution of a 
smooth function with an arbitrary distribution is smooth, it follows that the only non-smooth component of $K * f$
is given by $K_\delta * f_\delta$, which is supported in a ball of radius $2\delta$ around $x_0$. Since $\delta$ was
arbitrary, the statement follows. By linearity, this strongly suggests that the local structure of the singularities of $K*f$ can be described completely by only using knowledge on the local structure of the singularities of $f$.
It also suggests that the ``singular part'' of the operator $\CK$ should be local, with the non-local
parts of $\CK$ only contributing to the ``regular part''.

This discussion suggests that we certainly need the following ingredients to build an operator $\CK$ with the
desired properties:
\begin{claim}
\item The canonical polynomial structure should be part of our regularity structure in order to
be able to describe the ``regular parts''.
\item We should be given an ``abstract integration operator'' $\CI$ on $T$ which describes how the 
``singular parts'' of $\CR f$ transform under convolution by $K$.
\item We should restrict ourselves to models which are ``compatible'' with the action of $\CI$ in the sense
that the behaviour of $\Pi_x \CI \tau$ should relate in a suitable way to the behaviour of $K * \Pi_x \tau$ near $x$.
\end{claim}
One way to implement these ingredients is to assume first that our model space $T$ 
contains abstract polynomials in the following sense.

\begin{assumption}\label{ass:poly}
There exists a sector $\bar T \subset T$ isomorphic to the space of abstract polynomials
in $d$ commuting variables. In other words, $\bar T_\alpha \neq 0$ if and only if $\alpha \in \N$, 
and one can find basis vectors $X^k$ of $T_{|k|}$ such that every element $\Gamma \in G$ acts
on $\bar T$ by $\Gamma X^k = (X-h)^k$ for some $h \in \R^d$.
\end{assumption}

Furthermore, we assume that there exists an abstract integration operator $\CI$ with the following properties.

\begin{assumption}\label{ass:int}
There exists a linear map $\CI \colon T \to T$ such that $\CI T_\alpha \subset T_{\alpha + \beta}$,
such that $\CI \bar T = 0$, and such that, for every $\Gamma \in G$ and $\tau \in T$, one has
\begin{equ}[e:propI]
\Gamma \CI \tau - \CI \Gamma \tau \in \bar T\;.
\end{equ}
\end{assumption}

Finally, we want to consider models that are compatible with this structure for a given kernel $K$.
For this, we first make precise what we mean exactly when we said that $K$ is approximately homogeneous of 
degree $\beta - d$. 

\begin{assumption}\label{ass:kernel}
One can write $K = \sum_{n \ge 0} K_n$ where each of the kernels $K_n\colon \R^d \to \R$ is
smooth and compactly supported in a ball of radius $2^{-n}$ around the origin. Furthermore, we assume that
for every multiindex $k$, one has a constant $C$ such that the bound
\begin{equ}[e:boundKn]
\sup_x |D^k K_n(x)| \le C 2^{n(d-\beta + |k|)}\;,
\end{equ}
holds uniformly in $n$. Finally, we assume that $\int K_n(x) P(x)\,dx = 0$ for every polynomial $P$
of degree at most $N$, for some sufficiently large value of $N$.
\end{assumption}

\begin{remark}
It turns out that in order to define the operator $\CK$ on $\D^\gamma$, we will need $K$ to
annihilate polynomials of degree $N$ for some $N \ge \gamma + \beta$.
\end{remark}

\begin{remark}
The last assumption may appear to be extremely stringent at first sight. In practice, this turns out
not to be a problem at all. Say for example that we want to define an operator that represents convolution
with $\CG$, the Green's function of the Laplacian. Then, $\CG$ can be decomposed into a sum of terms
satisfying the bound \eref{e:boundKn} with $\beta = 2$,
but it does of course not annihilate generic polynomials and it is not supported in the ball of radius $1$.

However, for any fixed value of $N>0$, it is straightforward to decompose $\CG$ as $\CG = K + R$, where the
kernel $K$ is compactly supported and satisfies all of the properties mentioned above, and the kernel $R$ is
smooth. Lifting the convolution with $R$ to an operator from $\D^\gamma \to \D^{\gamma + \beta}$
(actually to $\D^{\bar \gamma}$ for any $\bar \gamma > 0$) is straightforward, so that we have reduced our
problem to that of constructing an operator describing the convolution by $K$.
\end{remark}

Given such a kernel $K$, we can now make precise what we meant earlier when we said that the models
under consideration should be compatible with the kernel $K$.

\begin{definition}
Given a kernel $K$ as in Assumption~\ref{ass:kernel} and a regularity structure $\TT$ satisfying
Assumptions~\ref{ass:poly} and \ref{ass:int}, we say that a model $(\Pi,\Gamma)$ is \textit{admissible} if 
the identities
\begin{equ}[e:defAdmissible]
\bigl(\Pi_x X^k\bigr)(y) = (y-x)^k\;,\qquad
\Pi_x \CI \tau = K * \Pi_x \tau - \Pi_x \CJ(x) \tau\;,
\end{equ}
holds for every $\tau \in T$ with $|\tau| \le N$. Here, $\CJ(x) \colon T \to \bar T$ is the linear map given on 
homogeneous elements by
\begin{equ}[e:defJ]
\CJ(x)\tau = \sum_{|k| < |\tau| + \beta} {X^k\over k!} \int D^{(k)} K(x-y)\,\bigl(\Pi_x \tau\bigr)(dy)\;.
\end{equ}
\end{definition}

\begin{remark}
Note first that if $\tau \in \bar T$, then the definition given above is coherent as long as $|\tau| < N$. Indeed,
since $\CI \tau = 0$, one necessarily has $\Pi_x \CI \tau = 0$. On the other hand, the properties of $K$
ensure that in this case one also has $K * \Pi_x \tau = 0$, as well as $\CJ(x)\tau = 0$.
\end{remark}

\begin{remark}\label{rem:welldef}
While $K * \xi$ is well-defined for any distribution $\xi$, it is not so clear \textit{a priori} whether
the operator $\CJ(x)$ given in \eref{e:defJ} is also well-defined. It turns out that the axioms of a model
do ensure that this is the case. The correct way of interpreting \eref{e:defJ} is by  
\begin{equ}
\CJ(x)\tau = \sum_{|k| < |\tau| + \beta} \sum_{n \ge 0} {X^k\over k!} \bigl(\Pi_x \tau\bigr)\bigl(D^{(k)} K_n(x-\cdot)\bigr)\;.
\end{equ}
Note now that the scaling properties of the $K_n$ ensure that $2^{(\beta - |k|)n} D^{(k)} K_n(x-\cdot)$ is a 
test function that is localised around $x$ at scale $2^{-n}$. As a consequence, one has 
\begin{equ}
\bigl|\bigl(\Pi_x \tau\bigr)\bigl(D^{(k)} K_n(x-\cdot)\bigr)\bigr| \lesssim 2^{(|k| - \beta - |\tau|)n}\;,
\end{equ}
so that this expression is indeed summable as long as $|k| < |\tau| + \beta$. 
\end{remark}

\begin{remark}
As a matter of fact, it turns out that the above definition of an admissible model dovetails very nicely 
with our axioms defining a general model. Indeed, starting from \textit{any} regularity structure $\TT$, \textit{any}
model $(\Pi, \Gamma)$ for $\TT$, and a kernel
$K$ satisfying Assumption~\ref{ass:kernel}, it is usually possible to build a larger regularity structure $\hat \TT$
containing $\TT$ (in the ``obvious'' sense that $T \subset \hat T$ and the action of $\hat G$ on $T$ is compatible with 
that of $G$) and endowed with an abstract integration map $\CI$, as well as an admissible model $(\hat \Pi, \hat \Gamma)$ on 
$\hat \TT$ which reduces to $(\Pi,\Gamma)$ when restricted to $T$. 
See \cite{Regular} for more details.

The only exception to this rule arises
when the original structure $T$ contains some homogeneous element $\tau$ which does not represent a polynomial
and which is such that $|\tau| + \beta \in \N$. Since the bounds appearing both in the definition of a model and
in Assumption~\ref{ass:kernel} are only upper bounds, it is in practice easy to exclude such a situation by slightly
tweaking the definition of either the exponent $\beta$ or of the original regularity structure $\TT$. 
\end{remark}

With all of these definitions in place, we can finally build the operator $\CK \colon \D^\gamma\to \D^{\gamma + \beta}$
announced at the beginning of this section. Recalling the definition of $\CJ$ from \eref{e:defJ}, we set
\begin{equ}[e:defKf]
\bigl(\CK f\bigr)(x) = \CI f(x) + \CJ(x) f(x) + \bigl(\CN f\bigr)(x)\;,
\end{equ}
where the operator $\CN$ is given by
\begin{equ}[e:defN]
\bigl(\CN f\bigr)(x) = 
\sum_{|k| < \gamma + \beta} {X^k\over k!} \int D^{(k)} K(x-y)\,\bigl(\CR f - \Pi_x f(x)\bigr)(dy)\;.
\end{equ}
Note first that thanks to the reconstruction theorem, it is possible to verify that the right hand side of 
\eref{e:defN} does indeed make sense for every $f \in \D^\gamma$ in virtually the same way as in Remark~\ref{rem:welldef}.
One has:

\begin{theorem}
Let $K$ be a kernel satisfying Assumption~\ref{ass:kernel}, let $\TT = (A,T,G)$ be a regularity structure
satisfying Assumptions~\ref{ass:poly} and \ref{ass:int}, and let $(\Pi,\Gamma)$ be an admissible model for $\TT$.
Then, for every $f \in \D^\gamma$ with $\gamma \in (0,N-\beta)$ and $\gamma + \beta \not \in \N$, 
the function $\CK f$ defined in \eref{e:defKf} belongs to $\D^{\gamma + \beta}$ and
satisfies $\CR \CK f = K * \CR f$.
\end{theorem}

\begin{proof}
The complete proof of this result can be found in \cite{Regular} and will not be given here. Let us simply show that one has
indeed $\CR \CK f = K * \CR f$ in the particular case when our model consists of continuous functions so
that Remark~\ref{rem:continuousModel} applies. In this case, one has
\begin{equ}
\bigl(\CR \CK f\bigr)(x) = \bigl(\Pi_x (\CI f(x) + \CJ(x) f(x))\bigr)(x) + \bigl(\Pi_x\bigl(\CN f\bigr)(x)\bigr)(x)\;.
\end{equ}
As a consequence of \eref{e:defAdmissible}, the first term appearing in the right hand side of this expression is given by
\begin{equ}
\bigl(\Pi_x (\CI f(x) + \CJ(x) f(x))\bigr)(x) = \bigl(K * \Pi_x f(x)\bigr)(x)\;.
\end{equ}
On the other hand, the only term contributing to the second term is the one with $k = 0$ (which is always present
since $\gamma > 0$ by assumption) which then yields
\begin{equ}
\bigl(\Pi_x\bigl(\CN f\bigr)(x)\bigr)(x) = \int K(x-y)\,\bigl(\CR f - \Pi_x f(x)\bigr)(dy)\;.
\end{equ}
Adding both of these terms, we see that the expression $\bigl(K * \Pi_x f(x)\bigr)(x)$ cancels, leaving us with the desired result.
\end{proof}

\section{Application of the theory to semilinear SPDEs}

Let us now briefly explain how this theory can be used to make sense of solutions to 
very singular semilinear stochastic PDEs. We will keep the discussion in this section at a very informal level
without attempting to make mathematically precise statements. The interested reader may find 
more details in \cite{Regular}.

For definiteness, we focus on the case of the dynamical
$\Phi^4_3$ model, which is formally given by
\begin{equ}[e:Phi43]
\d_t \Phi = \Delta \Phi - \Phi^3 + \xi\;,
\end{equ}
where $\xi$ denotes space-time white noise and the spatial variable takes values
in the $3$-dimen\-sio\-nal torus. The problem with such an equation is that even the solution to the 
linear part of the equation, namely 
\begin{equ}
\d_t \Psi = \Delta \Psi + \xi\;,
\end{equ}
only admits solutions in some spaces of Schwartz distributions. As a matter of fact, one has $\Psi(t,\cdot) \in \CC^{-\alpha}$
if and only if $\alpha > 1/2$.
As a consequence, it turns out that the only way of giving meaning to \eref{e:Phi43} is to ``renormalise'' the equation
by adding an ``infinite'' linear term ``$\infty \Phi$'' which counteracts the strong dissipativity of the term $-\Phi^3$.
To be slightly more precise, one can prove a statement of the following kind: 

\begin{theorem}\label{theo:Phi4}
Consider the sequence of equations
\begin{equ}[e:Phi43Regular]
\d_t \Phi_\eps = \Delta \Phi_\eps + C_\eps \Phi_\eps - \Phi_\eps^3 + \xi_\eps \;,
\end{equ}
where $\xi_\eps = \delta_\eps * \xi$ with $\delta_\eps(t,x) = \eps^{-5} \rho(\eps^{-2}t, \eps^{-1}x)$, for some
smooth and compactly supported function $\rho$, and $\xi$ denotes space-time white noise. 
Then, there exists a choice of constants $C_\eps$ such that the sequence $\Phi_\eps$ converges in probability
to a limiting (distributional) process $\Phi$. Furthermore, the limiting process $\Phi$ does \textit{not} depend
on the choice of mollifier $\rho$.
\end{theorem}

\begin{remark}
It turns out that in order to obtain a limit that is independent of the choice of mollifier $\rho$, 
one should take $C_\eps$ of the form
\begin{equ}
C_\eps = {c_1 \over \eps} + \tilde c \log \eps + c_3\;,
\end{equ}
where $\tilde c$ is universal, but $c_1$ and $c_3$ depend on the choice of $\rho$.
\end{remark}

\begin{remark}
The limiting solution $\Phi$ is only local in time, so that the precise statement has to
be slightly tweaked to allow for finite-time blow-ups. Regarding the initial condition,
one can take $\Phi_0 \in \CC^{-\beta}$ for any $\beta < 2/3$. This is expected to be
optimal, even for the deterministic equation. 
\end{remark}

The aim of this section is to sketch how the theory of regularity structures can be used to obtain
this kind of convergence results. First of all, we note that while our solution $\Phi$ will be 
a space-time distribution (or rather an element of $\D^\gamma$ for some regularity structure with 
a model over $\R^4$), the ``time'' direction has a different scaling behaviour from the three ``space''
directions. As a consequence, it turns out to be effective to slightly change our definition of
``localised test functions'' by setting
\begin{equ}
\phi_{(s,x)}^\lambda (t,y) = \lambda^{-5} \phi\bigl(\lambda^{-2}(t-s), \lambda^{-1}(y-x)\bigr)\;.
\end{equ}
Accordingly, the ``effective dimension'' of our space-time is actually $5$, rather than $4$.
The theory presented above extends \textit{mutatis mutandis} to this setting. (Note in particular
that when considering the degree of a regular monomial, powers of the time variable should now be 
counted double.) Note also that with this way of measuring regularity, space-time white noise
belongs to $\CC^{-\alpha}$ for every $\alpha > {5\over 2}$. This is because of the bound
\begin{equ}
\bigl(\E \scal{\xi, \phi_x^\lambda}^2\bigr)^{1/2} = \|\phi_x^\lambda\|_{L^2} \approx \lambda^{-{5\over 2}}\;,
\end{equ}
combined with an argument similar to the proof of Kolmogorov's continuity lemma.

\subsection{Construction of the associated regularity structure}

Our first step is to build a regularity structure that is sufficiently large to allow to 
reformulate \eref{e:Phi43} as a fixed point in $\D^\gamma$ for some $\gamma > 0$. 
Denoting by $\CG$ the heat kernel (i.e.\ the Green's function of the operator $\d_t - \Delta$), we 
can write the solution to \eref{e:Phi43} with initial condition $\Phi_0$ as
\begin{equ}
\Phi = \CG * \bigl(\xi - \Phi^3\bigr) + \CG \Phi_0\;,
\end{equ}
where $*$ denotes space-time convolution and where we denote by $\CG \Phi_0$ 
the harmonic extension of $\Phi_0$. In order to have a chance of fitting 
this into the framework described above, we first decompose the heat kernel $\CG$ as
\begin{equ}
\CG = K + \hat K\;,
\end{equ}
where the kernel $K$ satisfies all of the assumptions of Section~\ref{sec:Schauder} (with $\beta = 2$) and
the remainder $\hat K$ is smooth. If we consider any regularity structure containing the usual
Taylor polynomials and equipped with an admissible model, is straightforward to associate to $\hat K$ an
operator $\hat \CK \colon \D^\gamma \to \D^\infty$ via
\begin{equ}
\bigl(\hat \CK f\bigr)(z) = \sum_{k} {X^k \over k!} \bigl(D^{(k)}\hat K * \CR f\bigr)(z)\;,
\end{equ}
where $z$ denotes a space-time point and $k$ runs over all possible $4$-dimensional multiindices.
Similarly, the harmonic extension of $\Phi_0$ can be lifted to an element in $\D^\infty$
which we denote again by $\CG \Phi_0$ by
considering its Taylor expansion around every space-time point. At this stage, we note that we actually
cheated a little: while $\CG \Phi_0$ is smooth in $\{(t,x)\,:\, t > 0, x\in \T^3\}$ and vanishes when $t < 0$,
it is of course singular on the time-$0$ hyperplane $\{(0,x)\,:\, x\in \T^3\}$. This problem can be cured
by introducing weighted versions of the spaces $\D^\gamma$ allowing for singularities on
a given hyperplane. A precise definition of these spaces and their behaviour under multiplication and
the action of the integral operator $\CK$ can be found in \cite{Regular}. For the purpose of the informal
discussion given here, we will simply ignore this problem.

This suggests that the ``abstract'' formulation of \eref{e:Phi43} should be given by
\begin{equ}[e:abstractFull]
\Phi = \CK \bigl(\Xi - \Phi^3\bigr) + \hat \CK \bigl(\Xi - \Phi^3\bigr) + \CG \Phi_0\;.
\end{equ}
In view of \eref{e:defKf}, this equation is of the type
\begin{equ}[e:abstract]
\Phi = \CI \bigl(\Xi - \Phi^3\bigr) + (\ldots)\;,
\end{equ}
where the terms $(\ldots)$ consist of functions that take values in the subspace $\bar T$
of $T$ spanned by regular Taylor polynomials. 
In order to build a regularity structure in which \eref{e:abstract} can be formulated, 
it is natural to start with the structure given by abstract polynomials (again with the parabolic scaling
which causes the abstract ``time'' variable to have homogeneity $2$ rather than $1$), and to add
a symbol $\Xi$ to it which we postulate to have homogeneity $-{5\over 2}^{-}$, where we denote by $\alpha^-$ an
exponent strictly smaller than, but arbitrarily close to, the value $\alpha$.

We then simply add to $T$ all of the formal expressions that an application of the right hand side of \eref{e:abstract}
can generate for the description of $\Phi$, $\Phi^2$, and $\Phi^3$. The homogeneity of a given expression is 
completely determined by the rules $|\CI \tau| = |\tau| + 2$ and $|\tau \bar \tau| = |\tau| + |\tau|$.
More precisely, we consider a collection $\CU$ of formal expressions which is the
smallest collection containing $\one$, $X$, and $\CI(\Xi)$, and such that 
\begin{equ}
\tau_1,\tau_2,\tau_3 \in \CU \quad\Rightarrow\quad \CI(\tau_1\tau_2\tau_3) \in \CU\;,
\end{equ}
where it is understood that $\CI(X^k) = 0$ for every multiindex $k$.
We then set 
\begin{equ}
\CW = \{\Xi\} \cup \{\tau_1\tau_2\tau_3\,:\, \tau_i \in \CU\}\;,
\end{equ}
and we define our space $T$ as the set of all linear combinations of elements in $\CW$. 
(Note that since $\one \in \CU$, one does in particular have $\CU \subset \CW$.)
Naturally, $\CT_\alpha$ consists of those linear combinations that only involve
elements in $\CW$ that are of homogeneity $\alpha$. 
It is not too difficult to convince oneself that, 
for every $\alpha \in \R$, $\CW$ contains only
finitely many elements of homogeneity less than $\alpha$, so that each $\CT_\alpha$
is finite-dimensional.

In order to simplify expressions later, we will use the following shorthand graphical 
notation for elements of $\CW$. For $\Xi$, we simply draw a dot.
The integration map is then represented by a downfacing line and the multiplication of 
symbols is obtained by joining them at the root. For example, we have
\begin{equ}
\CI(\Xi) = \<1>\;,\quad
\CI(\Xi)^3 = \<3>\;,\quad
\CI(\Xi)\CI(\CI(\Xi)^3) = \<31>\;.
\end{equ}
Symbols containing factors of $X$ have no particular graphical representation, so we
will for example write $X_i \CI(\Xi)^2 = X_i\<2>$. With this notation, the space $T$ is 
given by
\begin{equs}
T &= \langle \Xi, \<3>, \<2>, \<32>,\<1>, \<31>, \<22>, X_i \<2>, \one, \<30>, \<21>, \ldots\rangle\;,
\end{equs}
where we ordered symbols in increasing order of homogeneity and used $\scal{\cdot}$ to denote
the linear span.
Given any sufficiently regular function $\xi$ (say a continuous space-time function), there is then
a canonical way of lifting $\xi$ to a model $\iota \xi = (\Pi,\Gamma)$ for $T$ by setting
\begin{equ}
\bigl(\Pi_x \Xi\bigr)(y) = \xi(y)\;,\qquad \bigl(\Pi_x X^k\bigr)(y) = (y-x)^k\;,
\end{equ}
and then recursively by
\begin{equ}[e:product]
\bigl(\Pi_x \tau \bar \tau\bigr)(y) = \bigl(\Pi_x \tau\bigr)(y)\cdot \bigl(\Pi_x \bar \tau\bigr)(y)\;,
\end{equ}
as well as \eref{e:defAdmissible}. (Note that here we used $x$ and $y$ as notations for generic
space-time points in order not to overload the notations.)

It turns out furthermore that there is a canonical way of building a structure group $G$ for $T$ and
to also lift $\xi$ to a family of operators $\Gamma_{xy}$, in such a way that 
all of the algebraic and analytic 
properties of an admissible model are satisfied. 
With such a model $\iota \xi$ at hand, it follows from \eref{e:product}
and the admissibility of $\iota \xi$ that the associated reconstruction operator satisfies the properties
\begin{equ}
\CR \CK f = K * \CR f \;,\qquad \CR (fg) = \CR f \cdot \CR g\;,
\end{equ} 
as long as all the functions to which $\CR$ is applied belong to $\D^\gamma$ for some $\gamma > 0$.
As a consequence, applying the reconstruction operator $\CR$ to both sides of  \eref{e:abstractFull},
we see that if $\Phi$ solves \eref{e:abstractFull} then, provided that the model
$(\Pi,\Gamma) = \iota\xi$ was built as above starting from any \textit{continuous} realisation $\xi$ of the driving noise, 
$\CR \Phi$ solves the equation \eref{e:Phi43}.

At this stage, the situation is as follows. For any \textit{continuous} realisation $\xi$ of the driving noise,
we have factored the solution map $(\Phi_0,\xi) \to \Phi$ associated to \eref{e:Phi43} into maps
\begin{equ}
(\Phi_0,\xi) \to (\Phi_0,\iota \xi) \to \Phi \to \CR \Phi \;,
\end{equ}
where the middle arrow corresponds to the solution to \eref{e:abstractFull} in some weighted $\D^\gamma$-space.
The advantage of such a factorisation is that the last two arrows yield \textit{continuous} maps, even in topologies
sufficiently weak to be able to describe driving noise having the lack of regularity of space-time white noise.
The only arrow that isn't continuous in such a weak topology is the first one. At this stage, it should be
believable that a similar construction can be performed for a very large class of semilinear stochastic PDEs.
In particular, the KPZ equation can also be analysed in this framework. 

Given this construction, one is lead naturally to the following 
question: given a sequence $\xi_\eps$ of ``natural'' regularisations
of space-time white noise, for example as in \eref{e:Phi43Regular}, do the lifts $\iota \xi_\eps$ converge in 
probably in a suitable space of admissible models? Unfortunately, unlike in the case of the theory of rough paths
where this is very often the case, 
the answer to this question in the context of SPDEs is often an emphatic \textbf{no}.
Indeed, if it were the case for the dynamical $\Phi^4_3$ model,
 then one could have chosen the constant $C_\eps$ to be independent of $\eps$ in 
\eref{e:Phi43Regular}, which is certainly not the case. 

\section{Renormalisation of the dynamical $\Phi^4_3$ model}

One way of circumventing the fact that $\iota\xi_\eps$ does not converge to a limiting model
as $\eps \to 0$ is to consider instead a sequence of \textit{renormalised} models. The main idea
is to exploit the fact that our abstract definitions of a model do not impose the identity
\eref{e:product}, even in situations where $\xi$ itself happens to be a continuous function. 
One question that then imposes itself is: what are the natural ways of ``deforming'' the usual
product which still lead to lifts to an admissible model? It turns out that the regularity structure
whose construction was sketched above comes equipped with a natural \textit{finite-dimensional} 
group of continuous transformations $\RR$ on its space of admissible models (henceforth called the 
``renormalisation group''), which essentially amounts
to the space of all natural deformations of the product. It then turns out that even though 
$\iota\xi_\eps$ does not converge, it is possible to find a sequence $M_\eps$ of elements in $\RR$ such that 
the sequence $M_\eps \iota \xi_\eps$ converges to a limiting model $(\hat \Pi, \hat \Gamma)$.
Unfortunately, the elements $M_\eps$ no \textit{not} preserve the image of $\iota$ in the space of
admissible models. As a consequence, when solving the fixed point map \eref{e:abstractFull} 
with respect to the model $M_\eps \iota \xi_\eps$ and inserting the solution into the reconstruction operator,
it is not clear \textit{a priori} that the resultong function (or distribution) can again be interpreted
as the solution to some modified PDE. It turns out that in our case, at least for a certain two-parameter subgroup of $\RR$,
this is again the case and the modified equation is precisely given by \eref{e:Phi43Regular}, where $C_\eps$ is 
some linear combination of the two constants appearing in the description of $M_\eps$.

There are now three questions that remain to be answered:
\begin{enumerate}
\item How does one construct the renormalisation group $\RR$?
\item How does one derive the new equation obtained when renormalising a model?
\item What is the right choice of $M_\eps$ ensuring that the renormalised models converge?
\end{enumerate}

\subsection{The renormalisation group}

In order to construct $\RR$, 
it is essential to first have some additional knowledge of the structure group $G$ for the type
of regularity structures considered above. 
Recall that the purpose of the group $G$ is to provide a class of linear maps $\Gamma \colon T \to T$
arising as possible candidates for the action of ``reexpanding'' a ``Taylor series'' around a different point.
In our case, in view of \eref{e:defAdmissible}, the coefficients of these reexpansions will naturally be
some polynomials in $x$ and in the expressions appearing in \eref{e:defJ}. This suggests that we should
define a space $T^+$ whose basis vectors consist of formal expressions of the type
\begin{equ}[e:genPolynom]
X^k \prod_{i=1}^N \CJ_{\ell_i}\tau_i\;,
\end{equ}
where $N$ is an arbitrary but finite number, the $\tau_i$ are basis elements of $T$,
and the $\ell_i$ are $d$-dimensional multiindices satisfying $|\ell_i| < |\tau_i| + 2$. 
(The last bound is a reflection of the restriction of the summands in \eref{e:defJ} with $\beta = 2$.)
The space $T^+$ also has a natural graded structure $T^+ = \bigoplus T^+_\alpha$ by setting
\begin{equ}
|\CJ_{\ell}\tau| = |\tau| + 2 - |\ell|\;,\qquad |X^k| = |k|\;,
\end{equ}
and by postulating that the degree of a product is the sum of the degrees. Unlike in the case of $T$ however,
elements of $T^+$ all have strictly positive homogeneity, except for the empty product $\one$ which we postulate
to have degree $0$.

To any given admissible model $(\Pi,\Gamma)$, it is then natural to associate linear maps
$f_x \colon T^+ \to \R$ by $f_x (X^k) = x^k$, $f_x(\sigma \bar \sigma) = f_x(\sigma) f_x(\bar \sigma)$, and
\begin{equ}[e:deffx]
f_x (\CJ_{\ell_i}\tau_i) = \int D^{(\ell_i)} K(x-y)\, \bigl(\Pi_x \tau_i\bigr)(dy)\;.
\end{equ}
It then turns out that it is possible to build a linear map $\Delta \colon T \to T \otimes T^+$ such that
if we define $F_x \colon T \to T$ by
\begin{equ}[e:defineaction]
F_x \tau = (I \otimes f_x)\Delta \tau\;,
\end{equ}
where $I$ denotes the identity operator on $T$, then these maps are invertible and
$\Pi_x F_x^{-1}$ is independent of $x$. Furthermore, there exists a map $\Delta^+ \colon T^+ \to T^+\otimes T^+$
such that 
\begin{equ}[e:propDelta]
(\Delta \otimes I)\Delta = (I \otimes \Delta^+)\Delta\;,\qquad \Delta^+(\sigma \bar \sigma) = \Delta^+ \sigma\cdot \Delta^+ \bar \sigma\;.
\end{equ}
With this map at hand, we can define a product $\circ$ on the space of linear functionals $f \colon T^+ \to \R$ by
\begin{equ}
(f \circ g)(\sigma) = (f \otimes g)\Delta^+ \sigma\;.
\end{equ}
If we furthermore denote by $\Gamma_f$ the operator $T$ associated to any such linear functional as in 
\eref{e:defineaction}, the first identity of \eref{e:propDelta} 
yields the identity $\Gamma_f \Gamma_g = \Gamma_{f\circ g}$. The second identity of \eref{e:propDelta} furthermore
ensures that if $f$ and $g$ are both multiplicative in the sense that $f(\sigma \bar \sigma) = f(\sigma) f(\bar \sigma)$,
then $f\circ g$ is again multiplicative. It also turns out that every multiplicative linear functional $f$ admits a 
unique inverse $f^{-1}$ such that $f^{-1} \circ f = f\circ f^{-1} = e$, where $e \colon T^+ \to \R$
maps every basis vector of the form \eref{e:genPolynom} to zero, except for $e(\one) = 1$.
The element $e$ is neutral in the sense that $\Gamma_e$ is the identity operator.

It is now natural to define the structure group $G$ associated to $T$ as the set of all multiplicative
linear functionals on $T^+$, acting on $T$ via \eref{e:defineaction}. Furthermore, for any admissible model,
one has the identity
\begin{equ}
\Gamma_{xy} = F_x^{-1} F_y = \Gamma_{\gamma_{xy}}\;,\qquad \gamma_{xy} = f_x^{-1} \circ f_y\;.
\end{equ}

How does all this help with the identification of a natural class of deformations for the usual product?
First, it turns out that for every continuous function $\xi$, if we denote again by $(\Pi,\Gamma)$ the model
$\iota \xi$, then the linear map $\PPi \colon T \to \CC$ given by
\begin{equ}
\PPi = \Pi_y F_y^{-1}\;,
\end{equ} 
which is independent of the choice of $y$ by the above discussion, 
is given by
\begin{equ}[e:idenPPi]
\bigl(\PPi \Xi\bigr)(x) = \xi(x)\;, \qquad
\bigl(\PPi X^k\bigr)(x) = x^k\;,
\end{equ}
and then recursively by
\begin{equ}
\PPi \tau \bar \tau = \PPi \tau \cdot \PPi \bar \tau\;,\qquad \PPi \CI \tau = K * \PPi \tau\;.
\end{equ}
Note that this is very similar to the definition of $\iota \xi$, with the notable exception
that \eref{e:defAdmissible} is replaced by the more ``natural'' identity $\PPi \CI \tau = K * \PPi \tau$.
It turns out that the knowledge of $\PPi$ and the knowledge of $(\Pi,\Gamma)$ are equivalent since one has
$\Pi_x = \Pi F_x$ and the map $F_x$ can be recovered from $\Pi_x$ by \eref{e:deffx}. (This argument appears 
circular but it is possible to put a suitable recursive structure on $T$ and $T^+$ ensuring that
this actually works.)
Furthermore, the translation $(\Pi,\Gamma) \leftrightarrow \PPi$ actually works for \textit{any} admissible model
and does not at all rely on the fact that it was built by lifting a continuous function. However, in the general case,
the first identity in \eref{e:idenPPi} does not of course not make any sense anymore and might fail even if the coordinates
of $\PPi$ consist of continuous functions. 

At this stage we note that if $\xi$ happens to be a stationary stochastic process and $\PPi$ is built
from $\xi$ by following the above procedure, then $\PPi \tau$ is a stationary stochastic process for every
$\tau \in T$. In order to define $\RR$, it is natural to consider only transformations of the space of admissible
models that preserve this property. Since we are not in general allowed to multiply components of $\PPi$,
the only remaining operation is to form linear combinations. It is therefore natural to describe elements of $\RR$ 
by linear maps $M \colon T \to T$ and to postulate their action on admissible models by $\PPi \mapsto \PPi^M$ with
\begin{equ}
\PPi^M \tau = \PPi M \tau\;.
\end{equ}
It is not clear \textit{a priori} whether given such a map $M$ and an admissible model
$(\Pi,\Gamma)$ there is a coherent way of building a new model $(\Pi^M, \Gamma^M)$ such that 
$\PPi^M$ is the map associated to $(\Pi^M, \Gamma^M)$ as above. It turns out that one has the following statement:

\begin{proposition}\label{prop:transform}
In the above context, for every linear map $M \colon T \to T$ commuting with $\CI$ and multiplication
by $X^k$, there exist \textit{unique}
linear maps $\Delta^M \colon T \to T \otimes T^+$ and $\hat \Delta^M \colon T^+ \to T^+ \otimes T^+$ such that if we set
\begin{equ}
\Pi_x^M \tau = \bigl(\Pi_x \otimes f_x\bigr)\Delta^M \tau \;,\qquad \gamma_{xy}^M(\sigma) = (\gamma_{xy} \otimes f_x)\hat \Delta^M\sigma\;,
\end{equ}
then $\Pi_x^M$ satisfies again \eref{e:defAdmissible} and the identity $\Pi_x^M \Gamma_{xy}^M = \Pi_y^M$.
\end{proposition}

At this stage it may look like \textit{any} linear map $M \colon T \to T$ commuting with $\CI$ and multiplication by $X^k$
yields a transformation on the space
of admissible models by Proposition~\ref{prop:transform}. This however is not true since we have completely
disregarded the \textit{analytical} bounds that every model has to satisfy.
It is clear from Definition~\ref{def:model} that these are satisfied if and only if $\Pi_x^M \tau$ is a linear combination
of the $\Pi_x \tau_j$ with $|\tau_j| \ge |\tau|$. This suggests the following definition. 

\begin{definition}
The renormalisation group $\RR$ consists of the set of linear maps $M \colon T \to T$ commuting with $\CI$
and with multiplication by $X^k$,
such that for $\tau \in T_\alpha$ and $\sigma \in T_\alpha^+$, one has 
\begin{equ}
\Delta^M \tau - \tau \otimes \one \in \bigoplus_{\beta > \alpha} T_\alpha \otimes T^+\;,\qquad
\hat \Delta^M \sigma - \sigma \otimes \one \in \bigoplus_{\beta > \alpha} T_\alpha^+ \otimes T^+\;.
\end{equ}
Its action on the space of admissible models is given by Proposition~\ref{prop:transform}.
\end{definition}

\subsection{The renormalised equations}

In the case of the dynamical $\Phi^4$ model considered in this article, 
it turns out that we only need a two-parameter subgroup of $\RR$
to renormalise the equations. More precisely, we consider elements $M \in \RR$ of the form
$M = \exp(- C_1 L_1 - C_2 L_2)$, where the two generators $L_1$ and $L_2$ are 
determined by the substitution rules
\begin{equ}
L_1 \colon \<2> \mapsto \one\;,\qquad L_2 \colon \<22> \mapsto \one\;.
\end{equ}
This should be understood in the sense that if $\tau$ is an arbitrary formal expression,
then $L_1 \tau$ is the sum of all formal expressions obtained from $\tau$ by performing
a substitution of the type $\<2> \mapsto 1$, and similarly for $L_2$. For example, one has
\begin{equ}
L_1 \<3> = 3 \<1>\;,\qquad L_1 \<12> = \<10> \;,\qquad L_2 \<32> = 3 \<1>\;.
\end{equ}
One then has the following result:

\begin{proposition}
The linear maps $M$ of the type just described belong to $\RR$. Furthermore, if $(\Pi,\Gamma)$ is an
admissible model such that $\Pi_x \tau$ is a continuous function for every $\tau \in T$, then one has the
identity
\begin{equ}[e:actionRenorm]
\bigl(\Pi_x^M \tau\bigr)(x) = \bigl(\Pi_x M \tau\bigr)(x)\;.
\end{equ}
\end{proposition}

\begin{remark}
Note that it it is the same value $x$ that appears twice on each side of \eref{e:actionRenorm}.
It is in fact \textit{not} the case that one has $\Pi_x^M \tau = \Pi_x M \tau$!
However, the identity \eref{e:actionRenorm} is all we need to derive the renormalised equations. 
\end{remark}

It is now rather straightforward to show the following:

\begin{proposition}
Let $M = \exp(- C_1 L_1 - C_2 L_2)$ as above and let $(\Pi^M,\Gamma^M) = M \iota \xi$ for some smooth function $\xi$.
Let furthermore $\Phi$ be the solution to \eref{e:abstractFull} with respect to the model $(\Pi^M,\Gamma^M)$. Then, the function
$u(t,x) = \bigl(\CR^M \Phi\bigr)(t,x)$ solves the equation
\begin{equ}
\d_t u = \Delta u - u^3 + (3C_1 - 9 C_2)u + \xi\;.
\end{equ}
\end{proposition}

\begin{proof}
By Theorem~\ref{theo:mult}, it turns out that \eref{e:abstractFull} can be solved in $\D^\gamma$ as soon as $\gamma$ is a
little bit greater than $1$. Therefore, we only need to keep track of its solution $\Phi$ up to terms of homogeneity $1$.
By repeatedly applying the identity \eref{e:abstract}, we see that the solution $\Phi$ is necessarily of the form
\begin{equ}[e:decompuPhi]
\Phi = \<1> + \phi\, \one - \<30> - 3 \phi\, \<20> + \scal{\nabla \phi, X}\;,
\end{equ}
for some real-valued function $\phi$ and some $\R^3$-valued function $\nabla \phi$. (Note that $\nabla \phi$ is treated
as an independent function here, we certainly do not suggest that the function $\phi$ is differentiable! Our notation
is only by analogy with the classical Taylor expansion...) Similarly, the right hand side of the equation is given up to
order $0$ by
\begin{equ}
\Xi - \Phi^3 = \Xi - \<3> - 3\phi\, \<2> + 3 \<32> - 3 \phi^2\,\<1> + 6 \phi\, \<31> 
 + 9\phi\, \<22> - 3 \scal{\nabla\phi, \<2> X}
- \phi^3\,\one \;. \label{e:RHS}
\end{equ}
Combining this with the definition of $M$, it is straightforward to see that, modulo terms of strictly
positive homogeneity, one has
\begin{equs}
M (\Xi - \Phi^3) &= \Xi - (M\Phi)^3 + 3C_1 \<1> + 3C_1 \phi \one - 9 C_2 \<1> - 9C_2 \phi \one\\
& = \Xi - (M\Phi)^3 + (3C_1 - 9C_2) M\Phi\;.
\end{equs}
Combining this with \eref{e:actionRenorm}, the claim now follows at once. 
\end{proof}

\subsection{Convergence of the renormalised models}

It remains to argue why one expects to be able to find constants $C_1^\eps$ and $C_2^\eps$ such that the 
sequence of renormalised models $M^\eps \iota \xi_\eps$ converges to a limiting model. 
Instead of considering the actual sequence of models, we only consider the sequence of
stationary processes $\hat \PPi^\eps \tau := \PPi^\eps M^\eps \tau$, where $\PPi^\eps$ is associated
to $(\Pi^\eps, \Gamma^\eps) = \iota \xi_\eps$ as before. Since there are general arguments available to deal
with all the expressions $\tau$ of positive homogeneity, we restrict ourselves to those of negative
homogeneity which, leaving out $\Xi$ which is easy to treat, are given by
\begin{equ}
\<3>,\; \<2>,\; \<32>,\;\<1>,\; \<31>,\; \<22>,\; X_i \<2>\;.
\end{equ} 

For this section, some elementary notions from the theory of Wiener chaos expansions are required, but we will
try to hide this as much as possible.
At a formal level, one has the identity
\begin{equ}
\PPi^\eps \<1> = K * \xi_\eps = K_\eps * \xi\;,
\end{equ}
where the kernel $K_\eps$ is given by $K_\eps = K * \delta_\eps$. This shows that, at least formally, one has
\begin{equ}
\bigl(\PPi^\eps \<2>\bigr)(z) = \bigl(K * \xi_\eps\bigr)(z)^2 = \int\!\!\int K_\eps(z-z_1)K_\eps(z-z_2)\, \xi(z_1)\xi(z_2)\,dz_1\,dz_2\;.
\end{equ}
Similar but more complicated expressions can be found for any formal expression $\tau$. 
This naturally leads to the study of random variables of the type
\begin{equ}[e:defIk]
I_k(f) = \int \!\cdots\! \int f(z_1,\ldots,z_k)\, \xi(z_1)\cdots \xi(z_k)\, dz_1\cdots dz_k\;.
\end{equ}
Ideally, one would hope to have an It\^o isometry of the type $\E I_k(f)I_k(g) = \scal{f^\sym,g^\sym}$,
where $\scal{\cdot,\cdot}$ denotes the $L^2$-scalar product and $f^\sym$ denotes the symmetrisation of $f$.
This is unfortunately \textit{not} the case. Instead, one should replace the products in \eref{e:defIk}
by \textit{Wick products}, which are formally generated by all possible \textit{contractions} of the type
\begin{equ}
\xi(z_i)\xi(z_j) \mapsto \xi(z_i)\diamond \xi(z_j) + \delta(z_i - z_j)\;.
\end{equ} 
If we then set
\begin{equ}
\hat I_k(f) = \int \!\cdots\! \int f(z_1,\ldots,z_k)\, \xi(z_1)\diamond\cdots \diamond\xi(z_k)\, dz_1\cdots dz_k\;,
\end{equ}
One has indeed
\begin{equ}
\E \hat I_k(f) \hat I_k(g) = \scal{f^\sym,g^\sym}\;.
\end{equ}
See \cite{Nualart} for a more thorough description of this construction, which also
goes under the name of \textit{Wiener chaos}.
It turns out that one has equivalence of moments in the sense that, for every $k>0$ and $p>0$ there exists a constant $C_{k,p}$
such that 
\begin{equ}
\E |\hat I_k(f)|^p \le C_{k,p} \|f^\sym\|^p\le C_{k,p} \|f\|^p\;,
\end{equ}
where the second bound comes from the fact that symmetrisation is a contraction in $L^2$.
Finally, one has $\E \hat I_k(f) \hat I_\ell(g) = 0$ if $k \neq \ell$. Random variables of the form $\hat I_k(f)$
for some $k \ge 0$ and some square integrable function $f$ are said to belong to the \textit{$k$th
homogeneous Wiener chaos}.

Returning to our problem, we first argue that it should be possible to choose $M$ in such a way that 
$\hat \PPi^\eps \<2>$ converges to a limit as $\eps \to 0$.
The above considerations suggest that one should rewrite $\PPi^\eps \<2>$ as
\begin{equ}[e:Psi2]
\bigl(\PPi^\eps \<2>\bigr)(z) = \bigl(K * \xi_\eps\bigr)(z)^2 = \int\!\!\int K_\eps(z-z_1)K_\eps(z-z_2)\, \xi(z_1)\diamond\xi(z_2)\,dz_1\,dz_2 + C_\eps\;,
\end{equ}
where the constant $C_\eps$ is given by
\begin{equ}
C_\eps = \int K_\eps^2(z_1)\,dz_1 = \int K_\eps^2(z-z_1)\,dz_1\;.
\end{equ}
Note now that $K_\eps$ is an $\eps$-approximation of the kernel $K$ which has the same singular behaviour
as the heat kernel. In terms of the parabolic distance, the singularity of the heat kernel scales like
$K(z) \sim |z|^{-3}$ for $z \to 0$. (Recall that we consider the parabolic distance $|(t,x)| = \sqrt{|t|} + |x|$,
so that this is consistent with the fact that the heat kernel is bounded by $t^{-3/2}$.)
This suggests that one has $K_\eps^2(z) \sim |z|^{-6}$ for $|z| \gg \eps$. Since parabolic space-time has scaling dimension
$5$ (time counts double!), this is a non-integrable singularity. As a matter of fact, there is a whole power of $z$
missing to make it borderline integrable, which suggests that one has
\begin{equ}
C_\eps \sim {1\over \eps}\;.
\end{equ}
This already shows that one should not expect $\PPi^\eps \<2>$ to converge to a limit as $\eps \to 0$. However,
it turns out that the first term in \eref{e:Psi2} converges to a distribution-valued stationary space-time process,
so that one would like to somehow get rid of this diverging constant $C_\eps$. This is exactly where the renormalisation
map $M$ (in particular the factor $\exp(-C_1 L_1)$) enters into play. Following the above
definitions, we see that one has
\begin{equ}
\bigl(\hat \PPi^\eps \<2>\bigr)(z) = \bigl(\PPi^\eps M\<2>\bigr)(z) = \bigl(\PPi^\eps \<2>\bigr)(z)  - C_1\;.
\end{equ}
This suggests that if we make the choice $C_1 = C_\eps$, then $\hat \PPi^\eps \<2>$ does indeed converge to a 
non-trivial limit as $\eps \to 0$. This limit is a distribution given by
\begin{equ}
\bigl(\PPi^\eps \<2>\bigr)(\psi) = \int\!\!\int \psi(z) K(z-z_1)K(z-z_2)\,dz\, \xi(z_1)\diamond\xi(z_2)\,dz_1\,dz_2\;.
\end{equ}
Using again the scaling properties of the kernel $K$,
it is not too difficult to show that this yields indeed a random variable belonging to the second homogeneous
Wiener chaos for every choice of smooth test function $\psi$. Once we know that $\hat \PPi^\eps \<2>$ converges,
it is immediate that $\hat \PPi^\eps X\<2>$ converges as well, since this amounts to just multiplying a distribution
by a smooth function. 

A similar argument to what we did for $\<2>$ allows to take care of $\tau = \<3>$ since one then has
\begin{equs}
\bigl(\PPi^\eps \<3>\bigr)(z) &= \int\!\!\int K_\eps(z-z_1)K_\eps(z-z_2)K_\eps(z-z_3)\, \xi(z_1)\diamond\xi(z_2)\diamond\xi(z_3)\,dz_1\,dz_2\,dz_3 \\
&\quad + 3 \int \!\!\int K_\eps(z-z_1)K_\eps(z-z_2)K_\eps(z-z_3) \delta(z_1 - z_2)\,\xi(z_3)\,dz_1\,dz_2\,dz_3\;.
\end{equs}
Noting that the second term in this expression is nothing but
\begin{equ}
3 C_\eps \int K_\eps(z-z_1)\,\xi(z_1)\,dz_1 = 3C_\eps \bigl(\PPi^\eps \<1>\bigr)(z)\;,
\end{equ}
we see that in this case, provided again that $C_1 = C_\eps$, 
$\hat \PPi^\eps \<3>$ is given by only the first term in the expression
above, which turns out to converge to a non-degenerate limiting random distribution in a similar way to what happened
for $\<2>$. 

Turning to our list of terms of negative homogeneity, it remains to consider $\<31>$, $\<22>$,
and $\<32>$. It turns out that the latter two are the more difficult ones, so we only discuss these.
Let us first argue why we expect to be able to choose the constants $C_1$ and $C_2$ in such a way that
$\hat \PPi^\eps \<22>$ converges to a limit. In this case, the ``bad'' terms comes from the 
part of $\bigl(\PPi^\eps \<22>\bigr)(z)$ belonging to the homogeneous chaos of order $0$.
This is simply a constant, which turns out to be given by
\begin{equ}[e:defC2]
\hat C_\eps = 2\int K(z) Q_\eps^2(z)\,dz\;,
\end{equ}
where the kernel $Q_\eps$ is given by
\begin{equ}
Q_\eps(z) = \int K_\eps(\bar z) K_\eps(\bar z - z)\,d\bar z\;.
\end{equ}
Since $K_\eps$ is an $\eps$-mollification of a kernel with a singularity of order $-3$ and the scaling dimension
of the underlying space is $5$, we see that $Q_\eps$ behaves like an $\eps$-mollification of a kernel
with a singularity of order $-3-3+5 = -1$ at the origin. As a consequence, the singularity of the integrand in \eref{e:defC2}
is of order $-5$, which gives rise to a logarithmic divergence as $\eps \to 0$. This suggests that 
one should choose $C_2 = \hat C_\eps$ in order to cancel out this diverging term and obtain a non-trivial limit
for $\hat \PPi^\eps \<22>$ as $\eps \to 0$. This is indeed the case. 

We finally turn to the symbol $\<32>$. In this case, the ``bad'' terms appearing in the Wiener
chaos decomposition of $\PPi^\eps \<32>$ are the terms in the first homogeneous Wiener chaos, which 
are of the form
\begin{equ}[e:firstHomogen]
3\int \hat Q_\eps(z-z_1) K_\eps(z_1- z_2)\xi(z_2)\,dz_1\,dz_2 = 3\int \bigl(\hat Q_\eps * K_\eps\bigr)(z- z_2)\xi(z_2)\,dz_2\;,
\end{equ}
where $\hat Q_\eps$ is the kernel given by
\begin{equ}
\hat Q_\eps(z) = 2 K(z) Q_\eps^2(z)\;.
\end{equ}
As already mentioned above, the problem here is that as $\eps \to 0$, $\hat Q_\eps$ converges to a kernel
$\hat Q = 2 K Q^2$, which has a non-integrable singularity at the origin. In particular, the action of integrating a
test function against $\hat Q_\eps$ does not converge to a limiting distribution as $\eps \to 0$. 

This is akin to the problem of making sense of integration against a one-dimensional kernel 
with a singularity of type $1/|x|$ at the origin. For the sake of the argument, let us consider a
function $W\colon \R \to \R$ which is compactly supported and smooth everywhere except at the origin,
where it diverges like $W(x) \sim 1/|x|$. It is then natural to associate to $W$ a ``renormalised'' distribution
$\Ren W$ given by
\begin{equ}
\bigl(\Ren W\bigr)(\phi) = \int W(x) \bigl(\phi(x) - \phi(0)\bigr)\,dx\;.
\end{equ}
Note that $\Ren W$ has the property that if $\phi(0) = 0$, then it simply corresponds to integration against $W$,
which is the standard way of associating a distribution to a function. In a way, the extra term can be interpreted
as subtracting a Dirac distribution with an ``infinite mass'' located at the origin, thus cancelling out the
divergence of the non-integrable singularity. It is also straightforward to verify that if $W_\eps$ is a sequence
of smooth approximations to $W$ (say one has $W_\eps(x) = W(x)$ for $|x| > \eps$ and $W_\eps \sim 1/\eps$ otherwise),
then $\Ren W^\eps \to \Ren W$ in a distributional sense, and (using the usual correspondence between functions
and distributions) one has
\begin{equ}
\Ren W^\eps = W^\eps - \hat C_\eps \delta_0\;,\qquad \hat C_\eps = \int W^\eps(x)\,dx\;.
\end{equ}
The cure to the problem we are facing for showing the convergence of $\PPi^\eps \<32>$ is
virtually identical. Indeed,by choosing $C_2 = \hat C_\eps$ as in \eref{e:defC2}, 
the term in the first homogeneous Wiener chaos for $\hat \PPi^\eps \<32>$ corresponding to 
\eref{e:firstHomogen} is precisely given by
\begin{equs}
3\int \hat Q_\eps(z-z_1) &K_\eps(z_1- z_2)\xi(z_2)\,dz_1\,dz_2 - 3C_2 \int K_\eps(z- z_2)\xi(z_2)\,dz_2 \\
& = 3\int \bigl(\Ren \hat Q_\eps * K_\eps\bigr)(z- z_2)\xi(z_2)\,dz_2\;.
\end{equs}
It turns out that the convergence of $\Ren \hat Q_\eps$ to a limiting distribution $\Ren \hat Q$ takes
place in a sufficiently strong topology to allow to conclude that $\hat \PPi^\eps \<32>$ 
does indeed converge to a non-trivial limiting random distribution.

It should be clear from this whole discussion that while the precise values of the constants $C_1$
and $C_2$ depend on the details of the mollifier $\delta_\eps$, the limiting (random)
model $(\hat \Pi,\hat \Gamma)$ obtained in this way is independent of it. Combining this with the
continuity of the solution to the fixed point map \eref{e:abstractFull} and of the reconstruction
operator $\CR$ with respect to the underlying model, we see that the statement of Theorem~\ref{theo:Phi4}
follows almost immediately.

\bibliographystyle{./Martin}
\bibliography{./refs}

\end{document}